\documentclass[a4paper,12pt,reqno]{amsart}

\usepackage{amsfonts}
\usepackage{amsmath}
\usepackage{amssymb}
\usepackage{cite}
\usepackage{mathrsfs}
\usepackage{enumitem}
\usepackage[colorlinks]{hyperref}
\usepackage[dvipsnames]{xcolor}
\numberwithin{equation}{section}

\usepackage{graphicx}

\newtheorem{theorem}{Theorem}[section]
\newtheorem{defi}[theorem]{Definition}
\newtheorem{remark}[theorem]{Remark}

\newtheorem{lemma}[theorem]{Lemma}

\def\Zn{{\mathbb Z}^n}

\def\Rn{{\mathbb R}^n}
\def\R2n{{\mathbb R}^{2n}}

\def\Rn{{\mathbb R}^n}

\def\R2{{\mathbb R}^2}
\def\R2n{{\mathbb R}^{2n}}

\def\N0{{\mathbb N}_{0}}

\def\l2h{{\ell^2(\hbar\mathbb Z^n)}}

\begin{document}

\title[Discrete wave equation for Schr\"{o}dinger operator ]
{Discrete time-dependent wave equation  for the Schr\"{o}dinger operator with unbounded potential}

\author[A. Dasgupta]{Aparajita Dasgupta}
\address{
	Aparajita Dasgupta:
	\endgraf
	Department of Mathematics
	\endgraf
	Indian Institute of Technology, Delhi, Hauz Khas
	\endgraf
	New Delhi-110016 
	\endgraf
	India
	\endgraf
	{\it E-mail address} {\rm adasgupta@maths.iitd.ac.in}
}
\author[S. S. Mondal]{Shyam Swarup Mondal}
\address{
	Shyam Swarup Mondal:
	\endgraf
	Department of Mathematics
	\endgraf
	Indian Institute of Technology, Delhi, Hauz Khas
	\endgraf
	New Delhi-110016 
	\endgraf
	India
	\endgraf
	{\it E-mail address} {\rm mondalshyam055@gmail.com}
}
\author[M. Ruzhansky]{Michael Ruzhansky}
\address{
	Michael Ruzhansky:
	\endgraf
	Department of Mathematics: Analysis, Logic and Discrete Mathematics
	\endgraf
	Ghent University, Belgium
	\endgraf
	and
	\endgraf
	School of Mathematical Sciences
	\endgraf
	Queen Mary University of London
	\endgraf
	United Kingdom
	\endgraf
	{\it E-mail address} {\rm ruzhansky@gmail.com}
}
\author[A. Tushir]{Abhilash Tushir}
\address{
	Abhilash Tushir:
	\endgraf
	Department of Mathematics
	\endgraf
	Indian Institute of Technology, Delhi, Hauz Khas
	\endgraf
	New Delhi-110016 
	\endgraf
	India
	\endgraf
	{\it E-mail address} {\rm abhilash2296@gmail.com}
}

\thanks{ The first and second authors were supported by Core Research Grant, RP03890G,  Science and Engineering
	Research Board (SERB), DST,  India. The third
 author was supported by the EPSRC Grants 
EP/R003025 and EP/V005529/1, by the FWO Odysseus 1 grant G.0H94.18N: Analysis and Partial Differential Equations and by the  Methusalem programme of the Ghent University Special Research Fund (BOF) (Grantnumber
01M01021). The last author is supported by the institute assistantship from Indian Institute of Technology Delhi, India. }
\date{\today}


\begin{abstract}
In this article, we investigate the semiclassical version of the wave equation for the discrete Schr\"{o}dinger operator, $\mathcal{H}_{\hbar,V}:=-\hbar^{-2}\mathcal{L}_{\hbar}+V$ on the lattice $\hbar\mathbb{Z}^{n},$  where $\mathcal{L}_{\hbar}$ is the discrete Laplacian, and $V$ is a non-negative multiplication operator. We prove that $\mathcal{H}_{\hbar,V}$ has a purely discrete spectrum when the potential $V$    satisfies the condition $|V(k)|\to \infty$ as $|k|\to\infty$. We also show that the Cauchy problem with regular coefficients is well-posed in the associated Sobolev type spaces and very weakly well-posed for distributional coefficients. Finally, we  recover the  classical solution as well as the very weak solution  in  certain
  Sobolev type spaces as the limit of the semiclassical parameter $\hbar\to 0$. 
\end{abstract}
\maketitle
\tableofcontents
\section{Introduction}\label{intro}
The aim of this paper is to study the semiclassical version of the wave equation with a discrete Schr\"{o}dinger operator on  the  lattice
\begin{equation*}
\hbar \mathbb{Z}^{n}=\left\{x \in \mathbb{R}^{n}: x=\hbar k, k \in \mathbb{Z}^{n}\right\},
\end{equation*}
depending on a (small) discretisation parameter $\hbar>0$, and the behaviour of its solutions as $\hbar \rightarrow 0$. Let  $v_j$ be the  $j^{t h}$  basis vector in  $\mathbb{Z}^{n}$, having all zeros except for  $1$  as the  $j^{t h}$  component. The discrete Schr\"{o}dinger operator on  $\hbar \mathbb{Z}^{n}$ denoted by $\mathcal{H}_{\hbar,V}$ is defined by
 \begin{equation}\label{dhamil}
 	\mathcal{H}_{\hbar,V}u(k):=\left(-\hbar^{-2}\mathcal{L}_{\hbar}+V\right)u(k),\quad k\in\hbar\mathbb{Z}^{n},
 \end{equation}
where the discrete Laplacian   $\mathcal{L}_{\hbar}$  is  given by
 \begin{equation}
 	\mathcal{L}_{\hbar} u(k):=\sum\limits_{j=1}^{n}\left(u\left(k+\hbar v_{j}\right)+u\left(k-\hbar v_{j}\right)\right)-2 n u(k),\quad k\in\hbar\mathbb{Z}^{n},
 \end{equation}
and $V$ is a non-negative multiplication operator by $V(k)$.  The semiclassical analogue of the classical wave equation with Schr\"{o}dinger operator is given by
\begin{equation}\label{dispde1}
	\left\{\begin{array}{l}
		\partial_{t}^{2} u(t, k)+a(t) \mathcal{H}_{\hbar,V}u(t, k)+q(t) u(t, k)=f(t,k),\quad(t, k) \in(0, T] \times \hbar\mathbb{Z}^{n}, \\
		u(0, k)=u_{0}(k),\quad k \in \hbar\mathbb{Z}^{n}, \\
		\partial_{t} u(0, k)=u_{1}(k),\quad k \in \hbar\mathbb{Z}^{n},
	\end{array}\right.
\end{equation} 
for the time-dependent propagation speed $a=a(t)\geq 0$, a bounded real-valued function $q$ and the  source term $f$. The equation \eqref{dispde1} is the semiclassical discretisation of the classical wave equation  with the Schr\"{o}dinger operator given by the following Cauchy problem
\begin{equation}\label{orgpde}
	\left\{\begin{array}{l}
		\partial_{t}^{2} u(t, x)+a(t) \mathcal{H}_{V}u(t, x)+q(t) u(t, x)=f(t,x),\quad(t, x) \in(0, T] \times \mathbb{R}^{n}, \\
		u(0, x)=u_{0}(x),\quad x \in \mathbb{R}^{n}, \\
		\partial_{t} u(0, x)=u_{1}(x),\quad x \in \mathbb{R}^{n},
	\end{array}\right.
\end{equation}
where $\mathcal{H}_{V}$ is the usual Schr\"{o}dinger operator on $\mathbb{R}^{n}$ with potential $V$ defined as
\begin{equation}\label{eucd}
	\mathcal{H}_{V}u(x):=\left(-\mathcal{L}+V\right)u(x),\quad x\in\mathbb{R}^{n},
\end{equation}
where $\mathcal{L}$ is the usual Laplacian on $\mathbb{R}^{n}$.

This family of linear operators for different potentials $V$ characterizes different quantum systems.  Let us look at some  critical quantum systems and their corresponding spectra.
  The Schr\"{o}dinger operator with free motion, (i.e.,  there is no force exerted on the electrons) is given by the Laplacian which has continuous spectrum contained in the positive real-axis.
The Schr\"{o}dinger operator of a hydrogen atom with an infinitely heavy nucleus placed at the origin   is given by the Coulomb potential, i.e.,
\begin{equation}
	\mathcal{H}=-\mathcal{L}+\frac{1}{|x|^{2}}.
\end{equation}
It has the essential spectrum contained in the positive half real-axis, and the spectrum on the negative real-axis contains isolated eigenvalues of finite multiplicity. The Schr\"{o}dinger operator with potential $V(x)=|x|^{2}$ and $V(x,y)=x^{2}y^{2}$, i.e., quantum harmonic oscillator and anharmonic oscillator, respectively, has a purely discrete spectrum. For more details about Schr\"{o}dinger operators, one can refer to    \cite{Gusta,Hislop,fern}. Moreover, we  have the following characterization of the potential for the purely discrete spectrum: 
\begin{theorem}\cite{Gusta}.\label{diseuc}
		Let $V :\mathbb{R}^{n}\to \mathbb{R}$ be continuous and satisfies $V(x)\geq0$, and $|V(x)|\to \infty$ as $|x|\to \infty$.
	Then $\sigma_{\mathrm{ess}}(\mathcal{H}_{V})=\emptyset$.
\end{theorem}
 There is limited literature available concerning the study of the spectrum of discrete Schr\"{o}dinger operator on $\hbar\mathbb{Z}^{n}$.  In \cite{Rab2004,Rab2006,Rab2009,Rab2010,Rab2013}, Rabinovich et. al. studied the Schr\"{o}dinger operator on the discrete lattice with slowly oscillating potential, and a few of them are dedicated  to the lattice $\hbar\mathbb{Z}^{n}$.  In \cite{swnkr}, Swain and Krishna addressed the purely discrete spectrum of the Schr\"{o}dinger operator on $\ell^{2}(\mathbb{Z}^{n})$ with potential $|k|^{\alpha}$, where $\alpha\in(0,1)$. 
 
 In the following theorem we will give the characterization of the discrete spectrum for the discrete Schr\"{o}dinger operator.  
\begin{theorem}\label{mdisc}
	Let $\hbar>0$.	Assume that $V\geq 0$ and $\left|V(k)\right|\to \infty$ as $\left|k\right|\to \infty$. Then $\mathcal{H}_{\hbar,V}:=-\hbar^{-2}\mathcal{L}_{\hbar}+V$ has a purely discrete spectrum.
\end{theorem}

We refer to Section \ref{faho} for the proof of above theorem and for more details about the spectrum of the discrete Schr\"{o}dinger operator. Moreover, we will also review the crucial components of the global Fourier analysis that was  developed  in \cite{niyaz1,Niyaz02} and later on applied for the wave equation associated with the Landau Hamiltonian in \cite{MRniya1,MRniya2}.

 Coming back to our main interest of this article, the Cauchy problems of the form \eqref{orgpde} have been extensively studied by many researchers. For the case of regular coefficients and source term, one can refer to  works  \cite{Colombini-deGiordi-Spagnolo-Pisa-1979,Fer2,Fer3,Fer4}. For distributional irregularities, for example, taking $q$ to be the $\delta$-distribution when electric potential produces shocks, and one can also take $a$ to be a Heaviside function when the propagation speed is discontinuous.  Due to the impossibility of the product of distributions (see Schwartz impossibility result \cite{schimp}),  the formulation of the Cauchy problem \eqref{dispde1} in this case might be impossible in the distributional sense but we are able to develop the well-posedness in very weak sense which was first introduced by Garetto and the third author in \cite{GRweak} and later on implemented in \cite{GR1,GR2,GR3} for different physical models. 
The third author and Tokmagambetov in \cite{GR2,MRniya1,MRniya2} studied the wave equations associated with operators having purely discrete spectrum and also allowing coefficients to have distributional irregularities. More precisely, in the case of regular coefficients, they have proved that the above Cauchy problem is well-posed in   the Sobolev space $\mathrm{H}_{\mathcal{H}_{V}}^{s}$ associated with the Schr\"{o}dinger operator $\mathcal{H}_{V}$, that is,
\begin{equation}
	\mathrm{H}_{\mathcal{H}_{V}}^{s}:=\left\{f \in \mathcal{D}_{\mathcal{H}_{V}}^{\prime}\left(\mathbb{R}^{n}\right): \left(I+\mathcal{H}_{V}\right)^{s / 2} f \in L^{2}\left(\mathbb{R}^{n}\right)\right\},\quad s\in\mathbb{R},
\end{equation}
with the norm $\|f\|_{\mathrm{H}_{\mathcal{H}_{V}}^{s}}:=\|\left(I+\mathcal{H}_{V}\right)^{s / 2} f\|_{L^{2}(\mathbb{R}^{n})}$, where $\mathcal{D}_{\mathcal{H}_{V}}^{\prime}\left(\mathbb{R}^{n}\right)$ is the global space of distributions associated to $\mathcal{H}_{V}$. In particular for non-negative polynomial potentials, the relation between the Sobolev space  associated with Schr\"{o}dinger operator and usual Sobolev spaces can be understood using the inequalities obtained
by Dziubanski and Glowacki in the following theorem:
\begin{theorem}\cite{MR2511755}. Let  $P(x)$ be a nonnegative homogeneous elliptic polynomial on $\mathbb{R}^{n}$ and $V$ is a nonnegative
	polynomial potential.
 For  $1<p<\infty$ and $\alpha>0$ there exist constants $C_1, C_2>0$ such that
 \begin{equation}\label{imp}
\left\|P(D)^\alpha f\right\|_{L^p}+\left\|V^\alpha f\right\|_{L^p} \leq C_1\left\|(P(D)+V)^{\alpha} f\right\|_{L^p},
 \end{equation}
and
\begin{equation*}
\left\|(P(D)+V)^{\alpha} f\right\|_{L^p} \leq C_2\left\|\left(P(D)^\alpha+V^\alpha\right) f\right\|_{L^p},
\end{equation*}
for $f$ in the Schwartz class $\mathcal{S}(\mathbb{R}^{n})$.
\end{theorem}
Now applying the inequality \eqref{imp}  for Schr\"{o}dinger operator $\mathcal{H}_{V}$ with non-negative polynomial potential and using the fact that the Schwartz space $\mathcal{S}(\mathbb{R}^{n})$ is dense in $L^{2}(\mathbb{R}^{n})$, we deduce the following:
\begin{equation}\label{semb}
	\mathrm{H}_{\mathcal{H}_{V}}^{s}(\mathbb{R}^{n})\subseteq \mathrm{H}^{s}(\mathbb{R}^{n}),\quad s>0.
\end{equation}

To simplify the notation, throughout the paper we will be writing $A \lesssim B$ if there exists a constant $C$ independent of the appearing parameters such that $A \leq C B$ and we write that $a\in L_{m}^{\infty}([0,T])$, if $a\in L^{\infty}([0,T])$ is $m$-times differentiable with $\partial^{j}_{t}a\in L^{\infty}([0,T])$, for all $j=1,\dots,m$. 
    
   The well-posedness result for regular coefficients is given by the following theorem: 
\begin{theorem}\label{eucclass}
Let $s\in\mathbb{R}$ and $f \in L^{2}([0,T] ; \mathrm{H}_{\mathcal{H}_{V}}^{s})$. Assume that $a \in L_{1}^{\infty}([0,T])$ and $q\in L^{\infty}([0,T])$  are such that $a(t) \geq a_{0}>0$ for some positive constant $a_{0}$. If the initial Cauchy data $\left(u_{0}, u_{1}\right) \in \mathrm{H}_{\mathcal{H}_{V}}^{1+s} \times \mathrm{H}_{\mathcal{H}_{V}}^{s}$, then the Cauchy problem \eqref{orgpde} has a unique solution $u \in C([0, T] ; \mathrm{H}_{\mathcal{H}_{V}}^{1+s}) \cap C^{1}([0, T] ; \mathrm{H}_{\mathcal{H}_{V}}^{s})$ which satisfies the estimate
	\begin{equation}
		\|u(t, \cdot)\|_{\mathrm{H}_{\mathcal{H}_{V}}^{1+s}}^{2}+\left\|\partial_{t} u(t, \cdot)\right\|_{\mathrm{H}_{\mathcal{H}_{V}}^{s}}^{2} \lesssim \left\|u_{0}\right\|_{\mathrm{H}_{\mathcal{H}_{V}}^{1+s}}^{2}+\left\|u_{1}\right\|_{\mathrm{H}_{\mathcal{H}_{V}}^{s}}^{2}+\|f\|_{L^{2}([0,T];\mathrm{H}_{\mathcal{H}_{V}}^{s})}^{2},
	\end{equation}
	with the constant independent of $t\in[0,T]$.
\end{theorem}
Furthermore, the Cauchy problem \eqref{orgpde} is very weakly well-posed in the case of distributional coefficients.
\begin{theorem}
	Let  $a$ and $q$  be  distributions with supports included in $[0, T]$ such that $a \geq a_0>0$ for some positive constant $a_{0}$, and also let the source term
	$f(\cdot, x)$ be a  distribution with support included in $[0, T]$, for all $x\in\mathbb{R}^{n}$. Let $s \in \mathbb{R}$ and the initial Cauchy data $(u_{0}, u_{1})\in\mathrm{H}_{\mathcal{H}_{V}}^{1+s}\times \mathrm{H}_{\mathcal{H}_{V}}^{s}$. Then the Cauchy problem \eqref{orgpde} has a very weak solution of order $s$.
\end{theorem}
\section{Main results}\label{mrcs}
 First, we  investigate the Cauchy problem \eqref{dispde1} with  regular coefficients  $a\in L_{1}^{\infty}([0,T])$ and $q\in L^{\infty}([0,T])$. We obtain the well-posedness in the discrete Sobolev space $\mathrm{H}_{\mathcal{H}_{\hbar,V}}^{s}$ associated with the discrete Schr\"{o}dinger operator $\mathcal{H}_{\hbar,V}$. Given $s\in\mathbb{R}$, we define the Sobolev space
\begin{equation}\label{sobo}
	 \mathrm{H}_{\mathcal{H}_{\hbar,V}}^{s}:=\left\{u \in \mathrm{H}^{-\infty}_{\mathcal{H}_{\hbar,V}} : \left(I+\mathcal{H}_{\hbar,V}\right)^{s / 2} u \in \ell^{2}\left(\hbar \mathbb{Z}^{n}\right)\right\},
\end{equation}
with the norm $\|f\|_{\mathrm{H}_{\mathcal{H}_{\hbar,V}}^{s}}:=\|\left(I+\mathcal{H}_{\hbar,V}\right)^{s / 2} f\|_{\ell^{2}(\hbar\mathbb{Z}^{n})}$, where $\mathrm{H}^{-\infty}_{\mathcal{H}_{\hbar,V}}$ is the space of $\mathcal{H}_{\hbar,V}$-distributions given in \eqref{hvdistr}. For a detailed study on  the Fourier analysis of discrete Schr\"{o}dinger operator and the associated Sobolev space, we refer to Section \ref{faho}.

The situation of the wave equation with the discrete Schr\"{o}dinger operator is in a striking difference with wave equation involving discrete Laplacian (fractional Laplacian)  that was considered in \cite{op1,op2}, respectively. The difference is in the sense that  we achieve the well-posedness in certain discrete Sobolev type spaces  in this case while  it was well-posed in $\ell^{2}(\hbar\mathbb{Z}^{n})$ in the later cases, and  this difference arises because of the difference in boundedness behavior of the discrete Schr\"{o}dinger operator  and the discrete Laplacian (fractional Laplacian) operator.   In the following theorem, we obtain the classical solution for the Cauchy problem \eqref{dispde1} with regular coefficients:
\begin{theorem}[Classical solution]\label{class}
	Let $T>0$. Let $s\in\mathbb{R}$ and  $f\in L^{2}([0,T];\mathrm{H}^{s}_{\mathcal{H}_{\hbar,V}})$. Assume that $a \in L_{1}^{\infty}\left([0, T]\right)$ satisfies $\inf\limits_{t \in[0, T]}a(t)= a_{0}>0$  and $q \in L^{\infty}\left([0, T]\right)$. If the initial Cauchy data  $\left(u_{0}, u_{1}\right) \in$
	$\mathrm{H}_{\mathcal{H}_{\hbar,V}}^{1+s} \times \mathrm{H}_{\mathcal{H}_{\hbar,V}}^{s}$, then the Cauchy problem \eqref{dispde1} has a unique solution $u \in$ $C([0, T]; \mathrm{H}_{\mathcal{H}_{\hbar,V}}^{1+s}) \cap C^{1}([0, T]; \mathrm{H}_{\mathcal{H}_{\hbar,V}}^{s})$ satisfying the estimate
	\begin{equation}\label{uestt}
		\|u(t,\cdot)\|^{2}_{\mathrm{H}_{\mathcal{H}_{\hbar,V}}^{1+s}}+\|u_{t}(t,\cdot)\|^{2}_{\mathrm{H}_{\mathcal{H}_{\hbar,V}}^{s}}\leq C_{T}\left(\|u_{0}\|^{2}_{\mathrm{H}_{\mathcal{H}_{\hbar,V}}^{1+s}}+\|u_{1}\|^{2}_{\mathrm{H}_{\mathcal{H}_{\hbar,V}}^{s}}+\|f\|^{2}_{L^{2}([0,T];\mathrm{H}_{\mathcal{H}_{\hbar,V}}^{s})}\right)	,
	\end{equation}
for all $t\in[0,T]$, where the constant $C_{T}$ is given by
\begin{equation}
	C_{T}=c_{0}^{-1}(1+\left\|a\right\|_{L^{\infty}})e^{c_{0}^{-1}\left(1+\|a^{\prime}\|_{L^{\infty}}+\|q\|_{L^{\infty}}+2\left\|a\right\|_{L^{\infty}}\right)T},
\end{equation}
with $c_{0}=\min\{a_{0},1\}$.
\end{theorem}
Furthermore, we will allow coefficients to be irregular and as we discussed earlier in Section \ref{intro}, we have a notion of very weak solutions  in order to handle equations that might not have a meaningful solution in the ordinary distributional sense.
 For the convenience of the  readers, we will quickly recap  the important details and state the corresponding results for $a,q\in\mathcal{D}^{\prime}\left([0,T]\right)$.
 Using the Friedrichs-mollifier, i.e., $\psi \in C_{0}^{\infty}\left(\mathbb{R}\right), \psi \geq 0$ and $\int_{\mathbb{R}} \psi=1$, we first regularise the distributional coefficient $a$ to obtain the families of smooth functions $\left(a_{\varepsilon}\right)_{\varepsilon},$ namely
\begin{equation*}\label{aeps}
	a_{\varepsilon}(t):=\left(a*\psi_{\omega(\varepsilon)}\right)\left(t\right),\quad \psi_{\omega(\varepsilon)}(t)=\dfrac{1}{\omega(\varepsilon)}\psi\left(\dfrac{t}{\omega(\varepsilon)}\right),\quad\varepsilon\in(0,1],
\end{equation*}
 where $\omega(\varepsilon)\geq 0$ and $\omega(\varepsilon)\to0$ as $\varepsilon\to 0$. 
\begin{defi}\label{cinfm}
	(i)	A net $\left(a_{\varepsilon}\right)_{\varepsilon} \in L^{\infty}_{m}(\mathbb{R})^{(0,1]}$ is said to be $L^{\infty}_{m}$-moderate if for all  $K \Subset \mathbb{R}$, there exist $N \in \mathbb{N}_{0}$ and $c>0$ such that
	\begin{equation*}	\left\|\partial^{k} a_{\varepsilon}\right\|_{L^{\infty}(K)} \leq c \varepsilon^{-N-k},\quad \text{ for all }k=0,1,\dots,m,
	\end{equation*}
	for all $\varepsilon \in(0,1]$.\\
	(ii) 	A net $\left(a_{\varepsilon}\right)_{\varepsilon} \in L^{\infty}_{m}(\mathbb{R})^{(0,1]}$ is said to be $L^{\infty}_{m}$-negligible if for all $K \Subset \mathbb{R}$ and   $q\in\mathbb{N}_{0}$, there exists  $c>0$ such that
	\begin{equation*}
	\left\|\partial^{k} a_{\varepsilon}\right\|_{L^{\infty}(K)} \leq c \varepsilon^{q},\quad \text{ for all }k=0,1,\dots,m,
	\end{equation*}
	for all $\varepsilon \in(0,1]$.\\
		(iii)	A net $\left(u_{\varepsilon}\right)_{\varepsilon} \in L^{2}([0,T];\mathrm{H}_{\mathcal{H}_{\hbar,V}}^{s})^{(0,1]}$ is said to be $L^{2}([0, T] ; \mathrm{H}_{\mathcal{H}_{\hbar,V}}^{s})$-moderate if   there exist $N \in \mathbb{N}_{0}$ and $c>0$ such that
	\begin{equation*}
		\|u_{\varepsilon}\|_{L^{2}([0,T];\mathrm{H}_{\mathcal{H}_{\hbar,V}}^{s})} \leq c \varepsilon^{-N},
	\end{equation*}
	for all $\varepsilon \in(0,1]$.\\
	(iv) A net $\left(u_{\varepsilon}\right)_{\varepsilon} \in L^{2}([0, T] ; \mathrm{H}_{\mathcal{H}_{\hbar,V}}^{s})^{(0,1]}$ is said to be $L^{2}([0, T] ; \mathrm{H}_{\mathcal{H}_{\hbar,V}}^{s})$-negligible if for all $q\in\mathbb{N}_{0}$ there exists $c>0$  such that
	\begin{equation*}
\|u_{\varepsilon}\|_{L^{2}([0,T];\mathrm{H}_{\mathcal{H}_{\hbar,V}}^{s})} \leq c \varepsilon^{q},
	\end{equation*}
	for all  $\varepsilon \in(0,1]$.
\end{defi}
We observe that the requirements of moderateness are natural in the sense that distributions are moderately regularised. Moreover, by the structure theorems for distributions, we have
	\begin{equation}\label{ccsd}
		\text{compactly supported distributions } \mathcal{E}^{\prime}(\mathbb{R}) \subset\left\{L^{2}\text{-moderate families}\right\}.
	\end{equation}
Therefore, the Cauchy problem \eqref{dispde1} may not have a solution in  compactly supported distributions  $\mathcal{E}^{\prime}(\mathbb{R})$. However, it may exist in the space of $L^{2}$-moderate families in some suitable sense.

Now, the notion of a very weak solution for the Cauchy problem \eqref{dispde1} can be introduced as follows:
\begin{defi}\label{vwkdef}
	Let $s \in \mathbb{R}, f \in L^{2}([0, T] ; \mathrm{H}^{s}_{\mathcal{H}_{\hbar,V}})$, and $(u_{0}, u_{1}) \in \mathrm{H}_{\mathcal{H}_{\hbar,V}}^{1+s}\times \mathrm{H}_{\mathcal{H}_{\hbar,V}}^{s}.$ The net $\left(u_{\varepsilon}\right)_{\varepsilon} \in$ $L^{2}([0, T] ; \mathrm{H}_{\mathcal{H}_{\hbar,V}}^{1+s})^{(0,1]}$ is a very weak solution of order $s$ of the Cauchy problem \eqref{dispde1} if there exist
	\begin{enumerate}
		\item $L^{\infty}_{1}$-moderate regularisation $a_{\varepsilon}$    of the coefficient $a$;
		\item $L^{\infty}$-moderate regularisation  $q_{\varepsilon}$ of the coefficient  $q$; and
		\item $L^{2}([0, T] ; \mathrm{H}_{\mathcal{H}_{\hbar,V}}^{s})$-moderate regularisation $f_{\varepsilon}$ of the source term $f$,
	\end{enumerate}
	such that $\left(u_{\varepsilon}\right)_{\varepsilon}$ solves the regularised problem
	\begin{equation}\label{reg}
		\left\{\begin{array}{l}
			\partial_{t}^{2} u_{\varepsilon}(t, k)+a_{\varepsilon}(t)\mathcal{H}_{\hbar,V}u_{\varepsilon}(t, k)+q_{\varepsilon}(t) u_{\varepsilon}(t, k)=f_{\varepsilon}(t, k),\quad(t, k) \in(0, T] \times \hbar\mathbb{Z}^{n}, \\
			u_{\varepsilon}(0, k)=u_{0}(k),\quad k \in \hbar\mathbb{Z}^{n}, \\
			\partial_{t} u_{\varepsilon}(0, k)=u_{1}(k),\quad k \in \hbar\mathbb{Z}^{n},
		\end{array}\right.
	\end{equation}
	for all $\varepsilon \in(0,1]$, and is $L^{2}([0, T] ; \mathrm{H}_{\mathcal{H}_{\hbar,V}}^{1+s})$-moderate.
\end{defi}
It should be noted that Theorem \ref{class} provides a unique solution to the regularised Cauchy problem \eqref{reg} that satisfies estimate \eqref{uestt}.

A distribution $a$ is said to be positive distribution if $\langle a,\psi\rangle\geq 0$, whenever $\psi \in C_0^{\infty}(\mathbb{R})$ satisfying $\psi\geq 0$, and distribution $a$ is said to be a
 strictly positive distribution if there exists a positive  constant $\alpha$ such that $a-\alpha$ is a positive distribution. In other words, $a\geq\alpha>0$, where $a\geq \alpha$  means that
\begin{equation*}
	\langle a-\alpha, \psi\rangle \geq 0, \quad \text{for all } \psi \in C_0^{\infty}(\mathbb{R}),\psi\geq 0.
\end{equation*}
Now we can state the existence theorem  for the
Cauchy problem \eqref{dispde1} with distributional coefficients as follows:
\begin{theorem}[Existence]\label{ext}
	Let   $a$ and $q$  be   distributions with supports contained in $[0, T]$ such that $a\geq a_{0}>0$ for some positive constant $a_{0}$, and also let  the source term $f(\cdot,k)$ be a distribution with  support contained in $[0,T]$, for all $k\in\hbar\mathbb{Z}^{n}$. Let $s\in\mathbb{R}$ and the initial Cauchy data  $\left(u_{0}, u_{1}\right) \in \mathrm{H}_{\mathcal{H}_{\hbar,V}}^{1+s} \times \mathrm{H}_{\mathcal{H}_{\hbar,V}}^{s}$. Then the Cauchy problem \eqref{dispde1} has a very weak solution of order s.
\end{theorem}
Furthermore, the uniqueness of the very weak solution can be interpreted as the negligible modification in the approximations of the coefficients $a,q,$ and the source term $f$,  has negligible impact on the family of very weak solutions. Formally the notion of uniqueness can be formulated as the authors did in \cite{marianna_schro12022,marianna_schro22022,marianna_heat2022}.
 \begin{defi}\label{uniquedef}
 	We say that the Cauchy problem \eqref{dispde1} has a $L^{2}([0, T] ; \mathrm{H}_{\mathcal{H}_{\hbar,V}}^{1+s})$-unique very weak solution, if
 	\begin{enumerate}
  \item for all  $L^{\infty}_{1}$-moderate nets $a_{\varepsilon},\tilde{a}_{\varepsilon}$ such that $(a_{\varepsilon}-\tilde{a}_{\varepsilon})_{\varepsilon}$ is $L^{\infty}_{1}$-negligible,
 		\item for all  $L^{\infty}$-moderate nets $q_{\varepsilon},\tilde{q}_{\varepsilon}$ such that $ (q_{\varepsilon}-\tilde{q}_{\varepsilon})_{\varepsilon}$ is $L^{\infty}$-negligible; and
 		\item for all $L^{2}([0, T] ; \mathrm{H}_{\mathcal{H}_{\hbar,V}}^{s})$-moderate nets $f_{\varepsilon},\tilde{f}_{\varepsilon}$ such that $(f_{\varepsilon}-\tilde{f}_{\varepsilon})_{\varepsilon}$ is \\ $L^{2}([0, T] ; \mathrm{H}_{\mathcal{H}_{\hbar,V}}^{s})$-negligible,
 	\end{enumerate}
 the net $(u_{\varepsilon}-\tilde{u}_{\varepsilon})_{\varepsilon}$ is $L^{2}([0, T] ; \mathrm{H}_{\mathcal{H}_{\hbar,V}}^{1+s})$-negligible, where
 	 $(u_{\varepsilon})_{\varepsilon}$ and $(\tilde{u}_{\varepsilon})_{\varepsilon}$ 	 are the families of
 	solutions corresponding to the 	$\varepsilon$-parametrised problems
 		\begin{equation}\label{reg1}
 		\left\{\begin{array}{l}
 			\partial_{t}^{2} u_{\varepsilon}(t, k)+a_{\varepsilon}(t)\mathcal{H}_{\hbar,V}u_{\varepsilon}(t, k)+q_{\varepsilon}(t) u_{\varepsilon}(t, k)=f_{\varepsilon}(t, k),\quad(t, k) \in(0, T] \times \hbar\mathbb{Z}^{n}, \\
 			u_{\varepsilon}(0, k)=u_{0}(k),\quad k \in \hbar\mathbb{Z}^{n}, \\
 			\partial_{t} u_{\varepsilon}(0, k)=u_{1}(k),\quad k \in \hbar\mathbb{Z}^{n},
 		\end{array}\right.
 	\end{equation}
 and
 	\begin{equation}\label{reg2}
 	\left\{\begin{array}{l}
 		\partial_{t}^{2} \tilde{u}_{\varepsilon}(t, k)+\tilde{a}_{\varepsilon}(t)\mathcal{H}_{\hbar,V}\tilde{u}_{\varepsilon}(t, k)+\tilde{q}_{\varepsilon}(t) \tilde{u}_{\varepsilon}(t, k)=\tilde{f}_{\varepsilon}(t, k),\quad(t, k) \in(0, T] \times \hbar\mathbb{Z}^{n}, \\
 		\tilde{u}_{\varepsilon}(0, k)=u_{0}(k),\quad k \in \hbar\mathbb{Z}^{n}, \\
 		\partial_{t} \tilde{u}_{\varepsilon}(0, k)=u_{1}(k),\quad k \in \hbar\mathbb{Z}^{n},
 	\end{array}\right.
 \end{equation}
respectively.
 \end{defi}
\begin{remark}
	The above notion of uniqueness can also be formulated in the sense of the Colombeau algebra $\mathcal{G}(\mathbb{R})$ which is defined in the following quotient form:
	\begin{equation}
		\mathcal{G}(\mathbb{R})=\frac{C^{\infty}\text {-moderate nets }}{C^{\infty}\text {-negligible nets }} .
	\end{equation}
	For more details about the Colombeau algebra, we refer to \cite{ober}. The uniqueness of very weak solutions in the sense of Colombeau algebra can be traced in various settings, see \cite{GRweak,MRniya1,GR2}.
\end{remark}
   The following theorem gives the
uniqueness of the very weak solution to the Cauchy problem \eqref{dispde1} in the  sense of Definition \ref{uniquedef}.
\begin{theorem}[Uniqueness]\label{uniq}
Let   $a$ and $q$  be  distributions with supports contained in $[0, T]$ such that $a\geq a_{0}>0$ for some positive constant $a_{0}$, and also let  the source term $f(\cdot,k)$ be a distribution with  support contained in $[0,T]$, for all $k\in\hbar\mathbb{Z}^{n}$. Let $s\in\mathbb{R}$ and the initial Cauchy data  $\left(u_{0}, u_{1}\right) \in \mathrm{H}_{\mathcal{H}_{\hbar,V}}^{1+s} \times \mathrm{H}_{\mathcal{H}_{\hbar,V}}^{s}$.	 Then the very weak solution of the Cauchy problem \eqref{dispde1} is $L^{2}([0,T];\mathrm{H}_{\mathcal{H}_{\hbar,V}}^{1+s})$-unique.
\end{theorem}
Now we give the consistency result, which means that very weak solutions in Theorem \ref{ext} recapture the
classical solutions given by Theorem \ref{class}, provided the latter exist.
\begin{theorem}[Consistency]\label{cnst}
Let $s\in \mathbb{R}$ and  $f \in L^{2}([0, T], \mathrm{H}_{\mathcal{H}_{\hbar,V}}^{s})$. Assume that $a \in L_{1}^{\infty}\left([0, T]\right)$ satisfies $\inf\limits_{t \in[0, T]}a(t)= a_{0}>0$  and $q \in L^{\infty}\left([0, T]\right)$. If the initial Cauchy data  $\left(u_{0}, u_{1}\right) \in \mathrm{H}_{\mathcal{H}_{\hbar,V}}^{1+s} \times \mathrm{H}_{\mathcal{H}_{\hbar,V}}^{s}$, then for any regularising families $a_{\varepsilon}, q_{\varepsilon}, f_{\varepsilon}$ in Definition \ref{vwkdef}, the very weak solution  $\left(u_{\varepsilon}\right)_{\varepsilon}$  converges to the classical solution of the Cauchy problem \eqref{dispde1} in $L^{2}([0, T]; \mathrm{H}_{\mathcal{H}_{\hbar,V}}^{1+s})$ as $\varepsilon \rightarrow 0$.
\end{theorem}
Furthermore, we are interested in approximating the classical solution in Euclidean settings by the solutions in discrete settings. The following theorem shows that under the assumption that the solutions
in $\mathbb{R}^{n}$ exist, they can be recovered in the limit as $\hbar\to 0$. We require a little additional Sobolev
regularity to ensure that the convergence results are global on the whole of $\mathbb{R}^{n}$. 
\begin{theorem}\label{semlimit}
Let $V$ be a non-negative polynomial potential in \eqref{eucd}.	Let $u$ and $v$ be the solutions of the Cauchy problems \eqref{dispde1} on $\hbar \mathbb{Z}^{n}$ and \eqref{orgpde} on $\mathbb{R}^{n}$, respectively, with the same Cauchy data  $u_{0}$ and $u_{1}$. Assume the initial Cauchy data $(u_{0},u_{1})\in \mathrm{H}_{\mathcal{H}_{V}}^{1+s}\times \mathrm{H}_{\mathcal{H}_{V}}^{s}$  with $s>4+\frac{n}{2}$ and also satisfying $(u_{0}^{(4v_{j})},u_{1}^{(4v_{j})})\in \mathrm{H}_{\mathcal{H}_{V}}^{1+s}\times \mathrm{H}_{\mathcal{H}_{V}}^{s}$ for all $j=1,\dots,n$. Then for every $t \in[0, T]$, we have
	\begin{equation}
		\|v(t)-u(t)\|_{\mathrm{H}_{\mathcal{H}_{\hbar,V}}^{1+s}}+\left\|\partial_{t} v(t)-\partial_{t} u(t)\right\|_{\mathrm{H}_{\mathcal{H}_{\hbar,V}}^{s}} \rightarrow 0, \text { as } \hbar \rightarrow 0,
	\end{equation}
and the convergence is uniform on $[0,T]$.
\end{theorem}
 Here the initial data and the source term of the Cauchy problem \eqref{dispde1} is the evaluation of the initial data and the source term from \eqref{orgpde} on the lattice $\hbar\mathbb{Z}^{n}$.
 
Similarly, in the semiclassical limit $\hbar\to 0$, the very weak solution of the Cauchy problem in the Euclidean setting can be approximated by the very weak solution in the discrete setting.
\begin{theorem}\label{vvyksemlimit}
Let $V$ be a non-negative polynomial potential in \eqref{eucd}. Let $(u_{\varepsilon})_{\varepsilon}$ and $(v_{\varepsilon})_{\varepsilon}$ be the very weak solutions of the Cauchy problems \eqref{dispde1} on $\hbar \mathbb{Z}^{n}$ and \eqref{orgpde} on $\mathbb{R}^{n}$, respectively, with the same Cauchy data $u_{0}$ and $u_{1}$. Assume the initial Cauchy data $(u_{0},u_{1})\in \mathrm{H}_{\mathcal{H}_{V}}^{1+s}\times \mathrm{H}_{\mathcal{H}_{V}}^{s}$  with $s>4+\frac{n}{2}$ and also satisfying $(u_{0}^{(4v_{j})},u_{1}^{(4v_{j})})\in \mathrm{H}_{\mathcal{H}_{V}}^{1+s}\times \mathrm{H}_{\mathcal{H}_{V}}^{s}$ for all $j=1,\dots,n$. Then for every $\varepsilon\in(0,1]$ and $t\in[0,T]$, we have
\begin{equation}
\|v_{\varepsilon}(t)-u_{\varepsilon}(t)\|_{\mathrm{H}_{\mathcal{H}_{\hbar,V}}^{1+s}}+\left\|\partial_{t} v_{\varepsilon}(t)-\partial_{t}u_{\varepsilon}(t)\right\|_{\mathrm{H}_{\mathcal{H}_{\hbar,V}}^{s}} \rightarrow 0 \text { as } \hbar \rightarrow 0,
\end{equation}
where the convergence is uniform on $[0,T]$ but pointwise for $\varepsilon\in(0,1]$.
\end{theorem}
\section{Spectrum and Fourier analysis of $\mathcal{H}_{\hbar,V}$}\label{faho}
  The  discrete Schr\"{o}dinger operator on  $\hbar \mathbb{Z}^{n}$ denoted by $\mathcal{H}_{\hbar,V}$  is defined by
\begin{equation}
	\mathcal{H}_{\hbar,V}u(k):=\left(-\hbar^{-2}\mathcal{L}_{\hbar}+V\right)u(k),\quad k\in\hbar\mathbb{Z}^{n},
\end{equation}
where $\mathcal{L}_{\hbar}$ is the discrete Laplacian   given by
\begin{equation}
	\mathcal{L}_{\hbar} u(k):=\sum\limits_{j=1}^{n}\left(u\left(k+\hbar v_{j}\right)+u\left(k-\hbar v_{j}\right)\right)-2 n u(k),\quad k\in\hbar\mathbb{Z}^{n},
\end{equation}
and the potential $V$ is a non-negative multiplication operator by $V(k)$.   We also note that $-\mathcal{L}_{\hbar}$ and $V$ are non-negative operators and so is $\mathcal{H}_{\hbar,V}$. 

Further, the operator  $\mathcal{H}_{\hbar,V}:\mathrm{Dom}\left(\mathcal{H}_{\hbar,V}\right)\to \ell^{2}(\hbar\mathbb{Z}^{n})$ is a densely defined and self-adjoint linear operator in $\ell^{2}(\hbar\mathbb{Z}^{n})$. The domain $\mathrm{Dom}\left(\mathcal{H}_{\hbar,V}\right)$ is given by
\begin{equation}
	\mathrm{Dom}\left(\mathcal{H}_{\hbar,V}\right):=\left\{u\in\ell^{2}(\hbar\mathbb{Z}^{n}):(I+\mathcal{H}_{\hbar,V})u\in\ell^{2}(\hbar\mathbb{Z}^{n})\right\}.
\end{equation}

We will now prove that $\mathcal{H}_{\hbar,V}$ has purely discrete spectrum; that is, $\sigma_{\mathrm{ess}}(\mathcal{H}_{\hbar,V})$ is an empty essential spectrum. It is well known to
be equivalent to proving that $\left(\mathcal{H}_{\hbar,V}+I\right)^{-1}$ is compact, i.e., $\mathcal{H}_{\hbar,V}$ has a compact resolvent (see \cite{edmund,Miklav}). Further, the set of eigenvectors will form a complete  orthonormal basis of $\ell^{2}(\hbar\mathbb{Z}^{n})$, since $(\mathcal{H}_{\hbar,V}+I)^{-1}$ is a compact operator.  In order to prove the compactness, we need the following preparatory lemma:
\begin{lemma}
	Let $V\geq 0$ be a multiplication operator satisfying $|V(k)|\to \infty$ as $|k|\to \infty$. Then $\left(V+I\right)^{-1}$ is compact.
\end{lemma}
\begin{proof}
Define an operator $Q$ by
\begin{equation}
	Qu(k):=\frac{1}{V(k)+1}u(k),\quad u\in\ell^{2}(\hbar\mathbb{Z}^{n}).
\end{equation}
It is obvious to verify that $(V+I)Qu=u$ and $Q(V+I)u=u$ which implies $(V+I)^{-1}=Q$. Moreover, $Q$ is a bounded linear operator on $\ell^{2}(\hbar\mathbb{Z}^{n})$, since
\begin{equation}
\left\|Qu\right\|_{\ell^{2}(\hbar\mathbb{Z}^{n})}^{2}=\sum_{k\in\hbar\mathbb{Z}^{n}}\frac{1}{|V(k)+1|^{2}}|u(k)|^{2}\leq \sum_{k\in\hbar\mathbb{Z}^{n}}|u(k)|^{2}=\left\|u\right\|_{\ell^{2}(\hbar\mathbb{Z}^{n})}^{2}.
\end{equation}
 The compactness of $Q$ will be shown by using the fact that it can be approximated by a  sequence of finite rank operators. Let $Q_{m}:\ell^{2}(\hbar\mathbb{Z}^{n})\to\ell^{2}(\hbar\mathbb{Z}^{n})$ be defined by
 \begin{equation}
 	Q_{m}u(k):=\left\{
 	\begin{array}{cc}
 		Qu(k), & |k|\leq m, \\
 		0, & \mathrm{otherwise}.
 	\end{array}\nonumber\right.
 \end{equation}
 Then $Q_{m}$ is bounded and a finite rank operator, implying that $Q_{m}$ is compact. Additionally,
 \begin{eqnarray}
 	\left\|(Q-Q_{m})u\right\|^{2}_{\ell^{2}(\hbar\mathbb{Z}^{n})}&=&\sum_{|k|>m}\frac{1}{|V(k)+1|^{2}}|u(k)|^{2}\nonumber\\\
 	&\leq&\Phi_{m}^{2}\left\|u\right\|^{2}_{\ell^{2}(\hbar\mathbb{Z}^{n})},
 \end{eqnarray}
where $\Phi_{m}:=\sup\left\{\frac{1}{|V(k)+1|}:|k|>m\right\}$ whence we have  
$\left\|(Q-Q_{m})\right\|\leq\Phi_{m}.$ Therefore
 $\left\|(Q-Q_{m})\right\|\to 0$, since $\Phi_{m}\to 0$ as $m\to\infty$. Hence $Q_{m}\to Q$ in norm. Thus $Q=(V+I)^{-1}$ is compact.
\end{proof} 
Now we are in position to prove that the discrete Schr\"{o}dinger operator $\mathcal{H}_{\hbar,V}$  has a purely discrete spectrum.
\begin{proof}[Proof of Theorem \ref{mdisc}]
Using the second resolvent identity for the discrete Schr\"{o}dinger operator and the potential $V$  at $\lambda=-1\in \rho\left(\mathcal{H}_{\hbar,V}\right)\cap \rho\left(V\right)$, we have
		\begin{equation}\label{2ri}
			\left(\mathcal{H}_{\hbar,V}+I\right)^{-1}	=\left(\mathcal{H}_{\hbar,V}+I\right)^{-1}\left(-\hbar^{-2}\mathcal{L}_{\hbar}\right)\left(V+I\right)^{-1}+\left(V+I\right)^{-1}.
		\end{equation}
	 The operators $\left(\mathcal{H}_{\hbar,V}+I\right)^{-1}$ and $	\left(V+I\right)^{-1}$ must be bounded, since $\lambda=-1$ belongs to their resolvent set. The boundedness of the discrete Laplacian on $\ell^{2}(\hbar\mathbb{Z}^{n})$ implies that $\left(\mathcal{H}_{\hbar,V}+I\right)^{-1}\left(-\hbar^{-2}\mathcal{L}_{\hbar}\right)$ is bounded. Being the product of a bounded and a compact operator, the right-hand side of \eqref{2ri} is compact. Hence $\left(\mathcal{H}_{\hbar,V}+I\right)^{-1}$ is compact.
\end{proof}

 Let us denote by $\sigma\left(\mathcal{H}_{\hbar,V}\right)=\{\lambda_{\xi}\geq 0:\xi\in\mathcal{I}_{\hbar}\}$  the discrete spectrum of the Schr\"{o}dinger operator $\mathcal{H}_{\hbar,V}$, where $\mathcal{I}_{\hbar}$ is a countable set and we arrange the eigenvalues in ascending order in accordance with the multiplicities that occur
\begin{equation*}\label{eigenorder}
	|\lambda_{\xi}|\leq |\lambda_{\eta}|,\text{ whenever } |\xi|\leq |\eta|,
\end{equation*}
for all  $\xi,\eta\in\mathcal{I}_{\hbar}$.  Furthermore, using the assumption $|V(k)|\to \infty$ as $|k|\to\infty$,  it is easy to conclude that the eigenvalues $\lambda_{\xi}$ of the discrete Schr\"{o}dinger operator $\mathcal{H}_{\hbar,V}$ tends to infinity as $|\xi|\to \infty$.

Let  $u_{\xi}$ be the eigenfunction  associated with the eigenvalue $\lambda_{\xi}$ for each $\xi\in\mathcal{I}_{\hbar}$,  i.e., 
\begin{equation}\label{eigen}
	\mathcal{H}_{\hbar,V}u_{\xi}=\lambda_{\xi}u_{\xi},\quad \text{for all } \xi\in \mathcal{I}_{\hbar}.
\end{equation} 
Consequently,   the  set of eigenvectors $\{u_{\xi}\}_{\xi\in\mathcal{I}_{\hbar}}$ forms an orthonormal basis for $\ell^{2}(\hbar\mathbb{Z}^{n})$, i.e.,
 \begin{equation}
\left(u_{\xi},u_{\eta}\right):=\left\{
	\begin{array}{cc}
		1, & \text{if }\xi=\eta, \\
		0, & \text{if }\xi\neq\eta,
	\end{array}\right.
\end{equation}
where 
\begin{equation}
	\left(f,g\right):=\sum_{k \in \hbar\mathbb{Z}^{n}}f(k)\overline{g(k)},
\end{equation}
is the usual inner product of the Hilbert space $\ell^{2}(\hbar\mathbb{Z}^{n})$.

Now, we will describe the spaces of distributions generated by $\mathcal{H}_{\hbar,V}$ and  its adjoint, and the
related global Fourier analysis. In our settings, there is a considerable reduction in complexity since the discrete Schr\"{o}dinger
operator is self-adjoint.

 The space $\mathrm{H}^{\infty}_{\mathcal{H}_{\hbar,V}}:=\operatorname{Dom}\left(\mathcal{H}_{\hbar,V}^{\infty}\right)$ is called the space of test functions for $\mathcal{H}_{\hbar,V}$,  defined by
\begin{equation*}
	\operatorname{Dom}\left(\mathcal{H}_{\hbar,V}^{\infty}\right):=\bigcap_{k=1}^{\infty} \operatorname{Dom}\left(\mathcal{H}_{\hbar,V}^{k}\right),
\end{equation*}
where $\operatorname{Dom}\left(\mathcal{H}_{\hbar,V}^{k}\right)$ is the domain of the operator $\mathcal{H}_{\hbar,V}^{k}$ defined as
\begin{equation*}
	\operatorname{Dom}\left(\mathcal{H}_{\hbar,V}^{k}\right):=\left\{u \in \ell^{2}(\hbar\mathbb{Z}^{n}): (I+\mathcal{H}_{\hbar,V})^{j} u \in \operatorname{Dom}(\mathcal{H}_{\hbar,V}), j=0,1, \ldots, k-1\right\}.
\end{equation*}
 The Fréchet topology of $\mathrm{H}^{\infty}_{\mathcal{H}_{\hbar,V}}$ is given by the family of norms
\begin{equation*}
\|u\|_{\mathrm{H}^{k}_{\mathcal{H}_{\hbar,V}}}:=\max _{j \leq k}\left\|(I+\mathcal{H}_{\hbar,V})^{j} u\right\|_{\ell^{2}(\hbar\mathbb{Z}^{n})}, \quad k \in \mathbb{N}_{0}, u \in \mathrm{H}^{\infty}_{\mathcal{H}_{\hbar,V}}.
\end{equation*}
The space 
\begin{equation}\label{hvdistr}
\mathrm{H}^{-\infty}_{\mathcal{H}_{\hbar,V}}:=\mathcal{L}\left(\mathrm{H}^{\infty}_{\mathcal{H}_{\hbar,V}}, \mathbb{C}\right),	
\end{equation}
 of all continuous linear functionals on $\mathrm{H}^{\infty}_{\mathcal{H}_{\hbar,V}}$ is called the space of $\mathcal{H}_{\hbar,V}$-distributions. For $w \in \mathrm{H}^{-\infty}_{\mathcal{H}_{\hbar,V}}$ and $u \in \mathrm{H}^{\infty}_{\mathcal{H}_{\hbar,V}}$, we shall write
\begin{equation*}
w(u)=( w, u) =\sum\limits_{k\in\hbar\mathbb{Z}^{n}}w(k)\overline{u(k)}.
\end{equation*}
For any $u \in \mathrm{H}^{\infty}_{\mathcal{H}_{\hbar,V}}$, the functional
\begin{equation*}
\mathrm{H}^{\infty}_{\mathcal{H}_{\hbar,V}} \ni v \mapsto (v,u),
\end{equation*}
is a $\mathcal{H}_{\hbar,V}$-distribution, which gives an embedding $u \in \mathrm{H}^{\infty}_{\mathcal{H}_{\hbar,V}} \hookrightarrow \mathrm{H}^{-\infty}_{\mathcal{H}_{\hbar,V}}.$ 
\begin{defi}[\textbf{Schwartz Space} $\mathcal{S}(\mathcal{I}_{\hbar})$]
	Let $\mathcal{S}(\mathcal{I}_{\hbar})$ denote the space of rapidly decaying functions $\varphi:\mathcal{I}_{\hbar}\to \mathbb{C}$, i.e.,  $\varphi \in \mathcal{S}(\mathcal{I}_{\hbar})$ if for any $M<\infty$, there exists a constant $C_{\varphi, M}$ such that
\begin{equation*}
	|\varphi(\xi)| \leq C_{\varphi, M}\langle\xi\rangle^{-M}, \quad\text{for all } \xi\in\mathcal{I}_{\hbar},
\end{equation*}
 where we denote
 \begin{equation*}
	\langle\xi\rangle:=\left(1+\lambda_{\xi}\right)^{\frac{1}{2 }}.
 \end{equation*}
\end{defi}
The topology on $\mathcal{S}(\mathcal{I}_{\hbar})$ is given by the family of seminorms $p_{k}$, where $k \in \mathbb{N}_{0}$ and
\begin{equation*}
p_{k}(\varphi):=\sup _{\xi \in \mathcal{I}_{\hbar}}\langle\xi\rangle^{k}|\varphi(\xi)| .
\end{equation*}
 We now define the $\mathcal{H}_{\hbar,V}$-Fourier transform on $\mathrm{H}^{\infty}_{\mathcal{H}_{\hbar,V}}$.
\begin{defi}
	We define the $\mathcal{H}_{\hbar,V}$-Fourier transform $	\mathcal{F}_{\mathcal{H}_{\hbar,V}}: \mathrm{H}^{\infty}_{\mathcal{H}_{\hbar,V}} \rightarrow \mathcal{S}(\mathcal{I}_{\hbar})$ by the formula
\begin{equation}
	\left(\mathcal{F}_{\mathcal{H}_{\hbar,V}} f\right)(\xi)=\widehat{f}(\xi):=\sum\limits_{k\in\hbar\mathbb{Z}^{n}} f(k) \overline{u_{\xi}(k)}, \quad \xi\in\mathcal{I}_{\hbar},
\end{equation}
where $u_{\xi}$ satisfies \eqref{eigen}.
\end{defi}
The above expression is well-defined by the H\"{o}lder inequality
\begin{equation*}
	\left|\widehat{f}(\xi)\right|\leq \left\|f\right\|_{\ell^{2}(\hbar\mathbb{Z}^{n})}\left\|u_{\xi}\right\|_{\ell^{2}(\hbar\mathbb{Z}^{n})}=\left\|f\right\|_{\ell^{2}(\hbar\mathbb{Z}^{n})}<\infty,
\end{equation*}
and then can be extended to $\mathrm{H}^{-\infty}_{\mathcal{H}_{\hbar,V}}$ in the usual way.
The $\mathcal{H}_{\hbar,V}$-Fourier transform $\mathcal{F}_{\mathcal{H}_{\hbar,V}}:\mathrm{H}^{\infty}_{\mathcal{H}_{\hbar,V}} \rightarrow \mathcal{S}(\mathcal{I}_{\hbar})$ is a homeomorphism and its  inverse 
\begin{equation*}
\mathcal{F}_{\mathcal{H}_{\hbar,V}}^{-1}:\mathcal{S}\left(\mathcal{I}_{\hbar}\right)\rightarrow	\mathrm{H}^{\infty}_{\mathcal{H}_{\hbar,V}},
\end{equation*}
is given by 
\begin{equation}
	\left(\mathcal{F}_{\mathcal{H}_{\hbar,V}}^{-1}g\right)(k):=\sum\limits_{\xi\in\mathcal{I}_{\hbar}}g(\xi)u_{\xi}(k),\quad g\in \mathcal{S}(\mathcal{I}_{\hbar}),
\end{equation}
 so that the $\mathcal{H}_{\hbar,V}$-Fourier inversion formula becomes
\begin{equation}
f(k)=\sum_{\xi \in \mathcal{I}_{\hbar}} \widehat{f}(\xi) u_{\xi}(k), \quad \text { for all } f \in \mathrm{H}^{\infty}_{\mathcal{H}_{\hbar,V}}.
\end{equation}
Consequently, the $\mathcal{H}_{\hbar,V}$-Plancherel formula takes the form
\begin{equation}\label{planch}
	\sum_{k \in \hbar\mathbb{Z}^{n}}|f(k)|^{2}	=\left(\sum_{\xi \in \mathcal{I}_{\hbar}} \widehat{f}(\xi) u_{\xi},\sum_{\eta \in \mathcal{I}_{\hbar}} \widehat{f}(\eta) u_{\eta}\right)=\sum\limits_{\xi,\eta\in\mathcal{I}_{\hbar}}\widehat{f}(\xi)\overline{\widehat{f}(\eta)}\left(u_{\xi},u_{\eta}\right)=\sum_{\xi \in \mathcal{I}_{\hbar}}|\widehat{f}(\xi)|^{2}.
\end{equation}
The  $\mathcal{H}_{\hbar,V}$-Fourier transform of the discrete Schr\"{o}dinger operator $\mathcal{H}_{\hbar,V}$ is
\begin{equation}\label{symbol}
	\left(\mathcal{F}_{\mathcal{H}_{\hbar,V}}\mathcal{H}_{\hbar,V}f\right)(\xi)=(\mathcal{H}_{\hbar,V}f,u_{\xi})=(f,\mathcal{H}_{\hbar,V}u_{\xi})=(f,\lambda_{\xi}u_{\xi})=\lambda_{\xi}\widehat{f}(\xi),\quad \xi\in\mathcal{I}_{\hbar},
\end{equation} 
since $\mathcal{H}_{\hbar,V}$ is a self-adjoint operator.
 Recalling \eqref{sobo}, \eqref{symbol} and the Plancherel's identity \eqref{planch}, we have
\begin{equation}\label{fnorm}
	\|f\|_{\mathrm{H}_{\mathcal{H}_{\hbar,V}}^{s}}:=\left\|(I+\mathcal{H}_{\hbar,V})^{s / 2} f\right\|_{\ell^{2}(\hbar\mathbb{Z}^{n})}=\left(\sum\limits_{\xi\in\mathcal{I}_{\hbar}}\langle\xi\rangle^{2s}\widehat{f}(\xi)|^{2}\right)^{\frac{1}{2}}=\left(\sum\limits_{\xi\in\mathcal{I}_{\hbar}}\left(1+\lambda_{\xi}\right)^{s}\widehat{f}(\xi)|^{2}\right)^{\frac{1}{2}}.
\end{equation}

\section{Proofs of the main results}\label{secmtp}
In this section, we will prove the existence of classical and very weak solutions of the Cauchy problem \eqref{dispde1} with regular and irregular coefficients, respectively. Further, we will prove the uniqueness and consistency for the very weak solutions.
\begin{proof}[Proof of Theorem \ref{class}]

Taking the $\mathcal{H}_{\hbar,V}$-Fourier transform of the Cauchy problem \eqref{dispde1} with respect to
$k\in\hbar\mathbb{Z}^{n}$ and using \eqref{symbol}, we obtain
\begin{equation}\label{dispde0}
	\left\{\begin{array}{l}
		\partial_{t}^{2} \widehat{u}(t, \xi)+a(t) \lambda_{\xi}\widehat{u}(t, \xi)+q(t)\widehat{u}(t, \xi)=\widehat{f}(t,\xi),\quad(t, \xi) \in(0, T] \times \mathcal{I}_{\hbar}, \\
		\widehat{u}(0, \xi)=\widehat{u}_{0}(\xi),\quad \xi \in \mathcal{I}_{\hbar}, \\
		\partial_{t} \widehat{u}(0,\xi)=\widehat{u}_{1}(\xi),\quad \xi \in \mathcal{I}_{\hbar}.
	\end{array}\right.
\end{equation}
 The basic idea of our further analysis is that we can investigate each equation in \eqref{dispde0} separately and then collect the estimates using the $\mathcal{H}_{\hbar,V}$-Plancherel formula \eqref{planch}. Thus, let us fix $\xi\in\mathcal{I}_{\hbar}$ and  use the  transformation
\begin{equation}
	U(t,\xi):=\left(\begin{array}{c}
		i\langle\xi\rangle \widehat{u}(t,\xi) \\
		\partial_{t}\widehat{u}(t,\xi)
	\end{array}\right), \quad U_{0}(\xi):=\left(\begin{array}{c}
		i \langle\xi\rangle\widehat{u}_{0}(\xi) \\
		\widehat{u}_{1}(\xi)
	\end{array}\right),		
\end{equation}
where $\langle \xi\rangle=(1+\lambda_{\xi})^{\frac{1}{2}}$, and the matrices
\begin{equation}
	A(t):=\left(\begin{array}{cc}
		 0 & 1\\
		 a(t) & 0
	\end{array}\right), \quad Q(t):=\left(\begin{array}{cc}
		 0 & 0 \\
		q(t)-a(t) & 0
	\end{array}\right) \text{ and }  F(t,\xi):=\left(\begin{array}{c}
		0 \\
		\widehat{f}(t,\xi)
	\end{array}\right).
\end{equation}	
This allows us to reformulate the given second order system \eqref{dispde0} as the first order system
\begin{equation}
	\left\{\begin{array}{l}
		\partial_{t}U(t,\xi)=i \langle\xi\rangle A(t) U(t,\xi)+i\langle\xi\rangle^{-1}Q(t) U(t,\xi)+F(t,\xi),\quad(t, \xi) \in(0, T] \times \mathcal{I}_{\hbar}, \\
		U(0,\xi)=U_{0}(\xi),\quad\xi \in \mathcal{I}_{\hbar}.
	\end{array}\right.
\end{equation}
We observe that the eigenvalues of the matrix $A(t)$ are given by $\pm \sqrt{a(t)}$.  The  symmetriser $S$ of matrix $A$ is given by 
\begin{equation}\label{sdef}
	S(t)=\left(\begin{array}{cc}
		 a(t) & 0\\
		0 & 1
	\end{array}\right),
\end{equation}
i.e., we have \begin{equation}
	SA-A^{*}S=0.
\end{equation}
Consider 
\begin{eqnarray}\label{srest}
	\left(S(t)U(t,\xi),U(t,\xi)\right)&=&a(t)\langle\xi\rangle^{2}\left|\widehat{u}(t,\xi)\right|^{2}+\left|\partial_{t}\widehat{u}(t,\xi)\right|^{2}\nonumber\\
	&\leq&\sup\limits_{t \in[0, T]}\{a(t),1\}\left(\langle\xi\rangle^{2}\left|\widehat{u}(t,\xi)\right|^{2}+\left|\partial_{t}\widehat{u}(t,\xi)\right|^{2}\right)\nonumber\\
	&=&\sup\limits_{t \in[0, T]}\{a(t),1\}\left|U(t,\xi)\right|^{2},
\end{eqnarray}
where $(\cdot, \cdot)$, and $|\cdot|$  denote the inner product and the norm in  $\mathbb{C}$, respectively. Similarly
\begin{equation}\label{slest}
	\left(S(t)U(t,\xi),U(t,\xi)\right)\geq\inf\limits_{t \in[0, T]}\{a(t),1\}\left|U(t,\xi)\right|^{2}.
\end{equation}
If we now define the energy
\begin{equation}
	E(t,\xi):=(S(t) U(t,\xi), U(t,\xi)),
\end{equation}
then from \eqref{srest}, and \eqref{slest}, it follows that
\begin{equation}\label{est}
	\inf_{t \in[0, T]}\{a(t), 1\}|U(t,\xi)|^{2} \leq E(t,\xi) \leq \sup _{t \in[0, T]}\{a(t), 1\}|U(t,\xi)|^{2} .
\end{equation}
Since $a\in L_{1}^{\infty}([0,T])$, there exist two positive constants $a_{0}$ and $a_{1}$ such that
\begin{equation}
\inf\limits_{t \in[0, T]} a(t)=a_{0}\quad \text{and } \sup\limits_{t \in[0, T]} a(t)=a_{1}. 
\end{equation}
 Further,
if we set $c_{0}=\min \left\{a_{0}, 1\right\}$ and $c_{1}=\max \left\{a_{1}, 1\right\}$, then the inequality \eqref{est} becomes
\begin{equation}\label{vEest}
	c_{0}|U(t,\xi)|^{2} \leq E(t,\xi) \leq c_{1}|U(t,\xi)|^{2},
\end{equation}
for all $t\in[0,T]$ and $\xi\in\mathcal{I}_{\hbar}$.
Then we can calculate
\begin{eqnarray}\label{vdest}
	E_{t}(t,\xi)&=&\left( S_{t}(t) U(t,\xi), U(t,\xi)\right)+\left(S(t)  U_{t}(t,\xi), U(t,\xi)\right) +\left(S(t) U(t,\xi),  U_{t}(t,\xi)\right) \nonumber\\
	&=&\left(S_{t}(t) U(t,\xi), U(t,\xi)\right)+i \langle\xi\rangle(S(t) A(t) U(t,\xi), U(t,\xi))+ \nonumber\\
	&&i\langle\xi\rangle^{-1}(S(t) Q(t) U(t,\xi), U(t,\xi))+\left(S(t) F(t,\xi),  U(t,\xi)\right)-\nonumber\\
	&&i \langle\xi\rangle(S(t) U(t,\xi), A(t) U(t,\xi)) -i\langle\xi\rangle^{-1}(S(t)  U(t,\xi), Q(t)U(t,\xi))+\nonumber\\
	&&\left(S(t), F(t,\xi) U(t,\xi)\right)\nonumber\\
	&=&\left(S_{t}(t) U(t,\xi), U(t,\xi)\right)+i \langle\xi\rangle\left(\left(S A-A^{*}S\right)(t) U(t,\xi), U(t,\xi)\right) +\nonumber\\
	&& i\langle\xi\rangle^{-1}\left(\left(S Q-Q^{*} S\right)(t) U(t,\xi), U(t,\xi)\right)+2\text{Re}\left(S(t)F(t,\xi), U(t,\xi)\right)\nonumber\\
	&=&\left(S_{t}(t) U(t,\xi), U(t,\xi)\right)+i\langle\xi\rangle^{-1}\left(\left(S Q-Q^{*} S\right)(t) U(t,\xi), U(t,\xi)\right)+ \nonumber\\
	&& 2\text{Re}\left(S(t)F(t,\xi), U(t,\xi)\right).
\end{eqnarray} 
From the definition of $S$ and $Q$, we have 
\begin{equation}
	S_{t}(t):=\left(\begin{array}{cc}
		 a^{\prime}(t) &0\\
		  0 & 0
	\end{array}\right)\text{ and }		\left(SQ-Q^{*}S\right)(t):=\left(\begin{array}{cc}
		0 & a(t)-q(t) \\
		q(t)-a(t) & 0
	\end{array}\right),
\end{equation}	
whence we get
\begin{equation}\label{sqst}
	\|S_{t}(t)\|\leq |a^{\prime}(t)|\quad\text{and}\quad \|\left(SQ-Q^{*}S\right)(t)\|\leq|q(t)|+|a(t)|,\quad \text{for all }t\in[0,T].
\end{equation}
Moreover, it is also obvious to observe that
\begin{equation}\label{snorm}
	\|S(t)\|\leq (1+|a(t)|),\quad \text{for all } t\in[0,T], 
\end{equation}
where $\|\cdot\|$ is the max norm. 
Combining the estimates \eqref{vEest}-\eqref{snorm} with the hypothesis $a\in L_{1}^{\infty}([0,T])$ and $q\in L^{\infty}([0,T])$, we get
\begin{eqnarray}\label{vetest}
	E_{t}(t,\xi)&\leq& \left\|S_{t}(t)\right\||U(t,\xi)|^{2}+\|\left(S Q-Q^{*}S\right)(t)\||U(t,\xi)|^{2}+2\|S(t)\||F(t,\xi)||U(t,\xi)|\nonumber\\
	&\leq&\left(\left\|S_{t}(t)\right\|+\|\left(S Q-Q^{*}S\right)(t)\|+\|S(t)\|\right)|U(t,\xi)|^{2}+\|S(t)\||F(t,\xi)|^{2}\nonumber\\	&\leq&\left(1+|a^{\prime}(t)|+|q(t)|+2|a(t)|\right)|U(t,\xi)|^{2}+(1+|a(t)|)|F(t,\xi)|^{2}\nonumber\\
	&\leq&\left(1+\|a^{\prime}\|_{L^{\infty}}+\|q\|_{L^{\infty}}+2\left\|a\right\|_{L^{\infty}}\right)|U(t,\xi)|^{2}+\left(1+\left\|a\right\|_{L^{\infty}}\right)|F(t,\xi)|^{2}\nonumber\\	&\leq&c_{0}^{-1}\left(1+\|a^{\prime}\|_{L^{\infty}}+\|q\|_{L^{\infty}}+2\left\|a\right\|_{L^{\infty}}\right)E(t,\xi)+\left(1+\left\|a\right\|_{L^{\infty}}\right)|F(t,\xi)|^{2}.\nonumber\\
\end{eqnarray}
If we set $\kappa_{1}=c_{0}^{-1}\left(1+\|a^{\prime}\|_{L^{\infty}}+\|q\|_{L^{\infty}}+2\left\|a\right\|_{L^{\infty}}\right)$ and $\kappa_{2}=1+\|a\|_{L^{\infty}}$, then  we get
\begin{equation}\label{vgrnineq}
	E_{t}(t,\xi)\leq\kappa_{1} E(t,\xi)+\kappa_{2}|F(t,\xi)|^{2}.
\end{equation}
By  using the  Gronwall's  lemma to the inequality \eqref{vgrnineq}, we deduce that  
\begin{equation}\label{vfest}
	E(t,\xi)\leq e^{\int_{0}^{t}\kappa_{1}\mathrm{d}\tau}\left(E(0,\xi)+\int_{0}^{t}\kappa_{2}|F(\tau,\xi)|^{2}\mathrm{d}\tau\right),
\end{equation}
for all $t\in[0,T]$ and $\xi\in\mathcal{I}_{\hbar}$.
Therefore by putting together   \eqref{vEest} and \eqref{vfest}, we obtain
\begin{multline}
	c_{0}|U(t,\xi)|^{2}\leq 	E(t,\xi)\leq e^{\int_{0}^{t}\kappa_{1}\mathrm{d}\tau}\left(E(0,\xi)+\int_{0}^{t}\kappa_{2}|F(\tau,\xi)|^{2}\mathrm{d}\tau\right)\\ \leq	 e^{\kappa_{1}T}\left(c_1|U(0,\xi)|^{2}+\kappa_{2}\int_{0}^{T}|F(\tau,\xi)|^{2}\mathrm{d}\tau\right).
\end{multline} 
This gives
\begin{equation}
	|U(t,\xi)|^{2}\leq C_{T}\left(|U(0,\xi)|^{2}+\int_{0}^{T}|F(\tau,\xi)|^{2}\mathrm{d}\tau\right),\quad (t,\xi)\in [0,T]\times\mathcal{I}_{\hbar},
\end{equation}
where $C_{T}=c_{0}^{-1}e^{\kappa_{1}T}\max\{c_{1},\kappa_{2}\}$. Then using the definition of $U$ and $F$, we obtain the inequality
\begin{equation}\label{mest}
	\langle\xi\rangle^{2}\left|\widehat{u}(t,\xi)\right|^{2}+\left|\partial_{t}\widehat{u}(t,\xi)\right|^{2} \leq C_{T}\left(\langle\xi\rangle^{2}\left|\widehat{u}_{0}(\xi)\right|^{2}+\left|\widehat{u}_{1}(\xi)\right|^{2}+\int_{0}^{T}|F(\tau,\xi)|^{2}\mathrm{d}\tau\right),
\end{equation}
with the constant independent of $t\in[0,T]$ and $\xi\in\mathcal{I}_{\hbar}$. More generally, multiplying \eqref{mest}  by powers of $\langle\xi\rangle$, for any $s\in\mathbb{R}$, we get
\begin{multline}\label{mmest}
\langle\xi\rangle^{2+2s}\left|\widehat{u}(t,\xi)\right|^{2}+\langle\xi\rangle^{2s}\left|\partial_{t}\widehat{u}(t,\xi)\right|^{2} \\\leq C_{T}\left(\langle\xi\rangle^{2+2s}\left|\widehat{u}_{0}(\xi)\right|^{2}+\langle\xi\rangle^{2s}\left|\widehat{u}_{1}(\xi)\right|^{2}+\langle\xi\rangle^{2s}\int_{0}^{T}|\widehat{f}(\tau,\xi)|^{2}\mathrm{d}\tau\right),
\end{multline}
i.e.,
\begin{multline}\label{mmmest}
	\left(1+\lambda_{\xi}\right)^{1+s}\left|\widehat{u}(t,\xi)\right|^{2}+\left(1+\lambda_{\xi}\right)^{s}\left|\partial_{t}\widehat{u}(t,\xi)\right|^{2} \\\leq C_{T}\left(\left(1+\lambda_{\xi}\right)^{1+s}\left|\widehat{u}_{0}(\xi)\right|^{2}+\left(1+\lambda_{\xi}\right)^{s}\left|\widehat{u}_{1}(\xi)\right|^{2}+\left(1+\lambda_{\xi}\right)^{s}\int_{0}^{T}|\widehat{f}(\tau,\xi)|^{2}\mathrm{d}\tau\right).
\end{multline}
Now, by using the $\mathcal{H}_{\hbar,V}$-Plancherel's formula \eqref{planch} and \eqref{fnorm}, we have
\begin{multline}
	\left\|(I+\mathcal{H}_{\hbar,V})^{\frac{1+s}{2}}u(t,\cdot)\right\|^{2}_{\ell^{2}(\hbar\mathbb{Z}^{n})}+\left\|(I+\mathcal{H}_{\hbar,V})^{\frac{s}{2}}u_{t}(t,\cdot)\right\|^{2}_{\ell^{2}(\hbar\mathbb{Z}^{n})}\\\leq		
	C_{T}\left(\left\|(I+\mathcal{H}_{\hbar,V})^{\frac{1+s}{2}}u_{0}\right\|^{2}_{\ell^{2}(\hbar\mathbb{Z}^{n})}+\left\|(I+\mathcal{H}_{\hbar,V})^{\frac{s}{2}}u_{1}\right\|^{2}_{\ell^{2}(\hbar\mathbb{Z}^{n})}+\right.\\
	\left.\left\|(I+\mathcal{H}_{\hbar,V})^{\frac{s}{2}}f\right\|^{2}_{L^{2}([0,T];\ell^{2}(\hbar\mathbb{Z}^{n}))}\right),
\end{multline}
whence we get
\begin{equation}
	\|u(t,\cdot)\|^{2}_{\mathrm{H}_{\mathcal{H}_{\hbar,V}}^{1+s}}+\|u_{t}(t,\cdot)\|^{2}_{\mathrm{H}_{\mathcal{H}_{\hbar,V}}^{s}}\leq		C_{T}\left(\|u_{0}\|^{2}_{\mathrm{H}_{\mathcal{H}_{\hbar,V}}^{1+s}}+\|u_{1}\|^{2}_{\mathrm{H}_{\mathcal{H}_{\hbar,V}}^{s}}+\|f\|^{2}_{L^{2}([0,T];\mathrm{H}_{\mathcal{H}_{\hbar,V}}^{s})}\right),
\end{equation}
for all $t\in[0,T]$, where the constant $C_{T}$ is given by
 \begin{equation}
 C_{T}=c_{0}^{-1}(1+\left\|a\right\|_{L^{\infty}})e^{c_{0}^{-1}\left(1+\|a^{\prime}\|_{L^{\infty}}+\|q\|_{L^{\infty}}+2\left\|a\right\|_{L^{\infty}}\right)T}.
 \end{equation}
 This completes the proof.
\end{proof}
Thus, we have obtained the well-posedness for the Cauchy problem \eqref{dispde1} in the Sobolev spaces associated with the discrete Schr\"{o}dinger operator. We will now prove the existence of very weak solution in the case of distributional coefficients.
\begin{proof}[Proof of Theorem \ref{ext}]
	Consider the regularised Cauchy problem
		\begin{equation}\label{vwprf}
		\left\{\begin{array}{l}
			\partial_{t}^{2} u_{\varepsilon}(t, k)+a_{\varepsilon}(t)\mathcal{H}_{\hbar,V}u_{\varepsilon}(t, k)+q_{\varepsilon}(t) u_{\varepsilon}(t, k)=f_{\varepsilon}(t, k),\quad(t, k) \in(0, T] \times \hbar\mathbb{Z}^{n}, \\
			u_{\varepsilon}(0, k)=u_{0}(k),\quad k \in \hbar\mathbb{Z}^{n}, \\
			\partial_{t} u_{\varepsilon}(0, k)=u_{1}(k),\quad k \in \hbar\mathbb{Z}^{n}.
		\end{array}\right.
	\end{equation}
 Taking the Fourier transform  with respect to $k\in\hbar\mathbb{Z}^{n}$ and then using the  transformation similar to Theorem \ref{class}, the Cauchy problem \eqref{vwprf} reduces to
	\begin{equation}
		\left\{\begin{array}{l}
			\partial_{t} U_{\varepsilon}(t,\xi)=i \langle\xi\rangle A_{\varepsilon}(t) U_{\varepsilon}(t,\xi)+i\langle\xi\rangle^{-1}Q_{\varepsilon}(t) U_{\varepsilon}(t,\xi)+F_{\varepsilon}(t,\xi),\quad \xi\in \mathcal{I}_{\hbar}, \\
			U_{\varepsilon}(0,\xi)=U_{0}(\xi),\quad \xi\in\mathcal{I}_{\hbar},
		\end{array}\right.
	\end{equation}
where
\begin{equation}
	U_{\varepsilon}(t,\xi):=\left(\begin{array}{c}
		i\langle\xi\rangle \widehat{u}_{\varepsilon}(t,\xi) \\
		\partial_{t}\widehat{u}_{\varepsilon}(t,\xi)
	\end{array}\right), \quad U_{0}(\xi):=\left(\begin{array}{c}
		i \langle\xi\rangle\widehat{u}_{0}(\xi) \\
		\widehat{u}_{1}(\xi)
	\end{array}\right),		
\end{equation}
and the matrices
\begin{equation}\label{aepqep}
	A_{\varepsilon}(t):=\left(\begin{array}{cc}
		 0 & 1\\
		 a_{\varepsilon}(t) & 0
	\end{array}\right),\text{ } Q_{\varepsilon}(t):=\left(\begin{array}{cc}
		 0 &0\\
		q_{\varepsilon}(t)-a_{\varepsilon}(t) & 0 
	\end{array}\right), \text{ } F_{\varepsilon}(t,\xi):=\left(\begin{array}{c}
		0 \\
		\widehat{f}_{\varepsilon}(t,\xi)
	\end{array}\right).
\end{equation}	
	We note that the eigenvalues of  $A_{\varepsilon}(t)$ are given by $\pm \sqrt{a_{\varepsilon}(t)}$. The  symmetriser $S_{\varepsilon}$  of $A_{\varepsilon}$ is given by 
	\begin{equation}\label{sepl}
		S_{\varepsilon}(t)=\left(\begin{array}{cc}	
			  a_{\varepsilon}(t) & 0\\
		0 & 1
		\end{array}\right),
	\end{equation}
	i.e., we have \begin{equation}
		S_{\varepsilon}A_{\varepsilon}-A_{\varepsilon}^{*}S_{\varepsilon}=0.
	\end{equation}
	If we now define the energy
	\begin{equation}
		E_{\varepsilon}(t,\xi):=(S_{\varepsilon}(t) U_{\varepsilon}(t,\xi), U_{\varepsilon}(t,\xi)),
	\end{equation}
	then similar to \eqref{est}, we have
	\begin{equation}
		\inf_{t \in[0, T]}\{a_{\varepsilon}(t), 1\}\left|U_{\varepsilon}(t,\xi)\right|^{2} \leq E_{\varepsilon}(t,\xi) \leq \sup_{t \in[0, T]}\{a_{\varepsilon}(t), 1\}\left|U_{\varepsilon}(t,\xi)\right|^{2} .
	\end{equation}
Recall that $a$ and $q$ are  distributions with  supports contained in $[0,T]$ and $\psi\in C_{0}^{\infty}(\mathbb{R}),\psi\geq 0$, $\text{supp}(\psi)\subseteq K.$ Given that the distributions $a$ and $q$ may be considered as supported in the interval $[0,T]$,  it is sufficient to assume $K=[0,T]$ throughout the article. By the structure theorem for compactly supported distributions, there exist $L_{1}, L_{2} \in \mathbb{N}$ and $c_{1}, c_{2}>0$ such that
\begin{equation}\label{std}
	\left|\partial_{t}^{k} a_{\varepsilon}(t)\right| \leq c_{1} \omega(\varepsilon)^{-L_{1}-k}\quad\text{and} \quad\left|\partial_{t}^{k} q_{\varepsilon}(t)\right| \leq c_{2} \omega(\varepsilon)^{-L_{2}-k},
\end{equation}
for all $k \in \mathbb{N}_{0}$ and $t \in[0, T]$. Since $a\geq a_{0}>0$ therefore we can write
\begin{equation}\label{amin}
	a_{\varepsilon}(t)=\left(a*\psi_{\omega(\varepsilon)}\right)(t)=\langle a,\tau_{t}\tilde{\psi}_{\omega(\varepsilon)}\rangle\geq\tilde{a}_{0}>0,
\end{equation}
where $\tilde{\psi}(x)=\psi(-x),x\in\mathbb{R}$ and $\tau_{t}\psi(\xi)=\psi(\xi-t),\xi\in\mathbb{R}$.
	Combining the inequalities \eqref{std} and \eqref{amin},  there exist two positive constants $c_{0}$ and $c_{1}$ such that
	\begin{equation}\label{vvEest}
		c_{0}\left|U_{\varepsilon}(t,\xi)\right|^{2} \leq E_{\varepsilon}(t,\xi) \leq \left(1+c_{1}\omega(\varepsilon)^{-L_{1}}\right)\left|U_{\varepsilon}(t,\xi)\right|^{2}.
	\end{equation}
Then we can calculate
	\begin{eqnarray}\label{vvdest}
		\partial_{t}E_{\varepsilon}(t,\xi)&=&
			\left(\partial_{t}S_{\varepsilon}(t) U_{\varepsilon}(t,\xi), U_{\varepsilon}(t,\xi)\right)+i \langle\xi\rangle\left(\left(S_{\varepsilon} A_{\varepsilon}-A^{*}_{\varepsilon} S_{\varepsilon}\right)(t) U_{\varepsilon}(t,\xi), U_{\varepsilon}(t,\xi)\right)+ \nonumber\\
			&&i\langle\xi\rangle^{-1}\left(\left(S_{\varepsilon} Q_{\varepsilon}-Q^{*}_{\varepsilon} S_{\varepsilon}\right)(t) U_{\varepsilon}(t,\xi), U_{\varepsilon}(t,\xi)\right)+2 \operatorname{Re}(S_{\varepsilon}(t) F_{\varepsilon}(t,\xi), U_{\varepsilon}(t,\xi))\nonumber\\
			&=&
			\left(\partial_{t}S_{\varepsilon}(t) U_{\varepsilon}(t,\xi), U_{\varepsilon}(t,\xi)\right)+i\langle\xi\rangle^{-1}\left(\left(S_{\varepsilon} Q_{\varepsilon}-Q^{*}_{\varepsilon} S_{\varepsilon}\right)(t) U_{\varepsilon}(t,\xi), U_{\varepsilon}(t,\xi)\right)+\nonumber\\&&
			2 \operatorname{Re}(S_{\varepsilon}(t) F_{\varepsilon}(t, \xi), U_{\varepsilon}(t,\xi))\nonumber\\ 
		&\leq&\left(\left\|\partial_{t}S_{\varepsilon}(t)\right\|+\|\left(S_{\varepsilon} Q_{\varepsilon}-Q^{*}_{\varepsilon}S_{\varepsilon}\right)(t)\|+\|S_{\varepsilon}(t)\|\right)|U_{\varepsilon}(t,\xi)|^{2}+\nonumber\\
		&&\|S_{\varepsilon}(t)\||F_{\varepsilon}(t,\xi)|^{2}\nonumber\\
		&\leq& \left(1+|a_{\varepsilon}^{\prime}(t)|+|q_{\varepsilon}(t)|+2|a_{\varepsilon}(t)|\right)|U_{\varepsilon}(t,\xi)|^{2}+\left(1+|a_{\varepsilon}(t)|\right)|F_{\varepsilon}(t,\xi)|^{2}.\nonumber\\
	\end{eqnarray}
 Combining the above estimates with \eqref{std} and \eqref{vvEest}, we obtain
\begin{eqnarray}\label{gnwprest}
	\partial_{t}E_{\varepsilon}(t,\xi)&\leq& c_{0}^{-1}\left(1+c_{1}\omega(\varepsilon)^{-L_{1}-1}+c_{2}\omega(\varepsilon)^{-L_{2}}+2c_{1}\omega(\varepsilon)^{-L_{1}}\right)E_{\varepsilon}(t,\xi)+\nonumber\\&&\left(1+c_{1}\omega(\varepsilon)^{-L_{1}}\right)|F_{\varepsilon}(t,\xi)|^{2}\nonumber\\
	&=&\kappa_{1}\left(1+\omega(\varepsilon)^{-L_{1}-1}+\omega(\varepsilon)^{-L_{2}}+\omega(\varepsilon)^{-L_{1}}\right)E_{\varepsilon}(t,\xi)+\nonumber\\&&\kappa_{2}\left(1+\omega(\varepsilon)^{-L_{1}}\right)|F_{\varepsilon}(t,\xi)|^{2},
\end{eqnarray}
where $\kappa_{1}=c_{0}^{-1}\max\{1,2c_{1},c_{2}\}$ and $\kappa_{2}=\max\{1,c_{1}\}$. Applying the Gronwall's lemma to the inequality \eqref{gnwprest}, we obtain
\begin{multline}\label{gnvest}
E_{\varepsilon}(t,\xi)\leq e^{\int_{0}^{t}\kappa_{1}\left(1+\omega(\varepsilon)^{-L_{1}-1}+\omega(\varepsilon)^{-L_{2}}+\omega(\varepsilon)^{-L_{1}}\right)\mathrm{d}\tau}\times\\
\left(E_{\varepsilon}(0,\xi)+\kappa_{2}\left(1+\omega(\varepsilon)^{-L_{1}}\right)\int_{0}^{t}|F_{\varepsilon}(\tau,\xi)|^{2}\mathrm{d}\tau\right).
\end{multline}
 Combining the inequalities \eqref{vvEest} and \eqref{gnvest}, we obtain
\begin{eqnarray}
	c_{0}\left|U_{\varepsilon}(t,\xi)\right|^{2}&\leq&
E_{\varepsilon}(t,\xi)\nonumber\\
&\leq&  e^{\kappa_{1}\left(1+\omega(\varepsilon)^{-L_{1}-1}+\omega(\varepsilon)^{-L_{2}}+\omega(\varepsilon)^{-L_{1}}\right)T}\times\nonumber\\
&&\left(\left(1+c_{1}\omega(\varepsilon)^{-L_{1}}\right)\left|U_{\varepsilon}(0,\xi)\right|^{2}+\kappa_{2}\left(1+\omega(\varepsilon)^{-L_{1}}\right)\int_{0}^{T}|F_{\varepsilon}(\tau,\xi)|^{2}\mathrm{d}\tau\right)\nonumber\\
&\leq&C_{T}e^{\kappa_{T}\left(\omega(\varepsilon)^{-L_{1}-1}+\omega(\varepsilon)^{-L_{2}}+\omega(\varepsilon)^{-L_{1}}\right)}\left(\left|U_{\varepsilon}(0,\xi)\right|^{2}+\int_{0}^{T}|F_{\varepsilon}(\tau,\xi)|^{2}\mathrm{d}\tau\right),\nonumber\\
\end{eqnarray}
where $C_T=e^{k_{1}T}\max\{1,c_{1},\kappa_{2}\}$ and $\kappa_{T}=2+2k_{1}T$. This gives
\begin{multline}\label{vineq}
	|U_{\varepsilon}(t,\xi)|^{2}\leq c_{0}^{-1}C_{T}e^{\kappa_{T}\left(\omega(\varepsilon)^{-L_{1}-1}+\omega(\varepsilon)^{-L_{2}}+\omega(\varepsilon)^{-L_{1}}\right)}\left(\left|U_{\varepsilon}(0,\xi)\right|^{2}+\int_{0}^{T}|F_{\varepsilon}(\tau,\xi)|^{2}\mathrm{d}\tau\right).
\end{multline}
Putting $\omega(\varepsilon)\sim|\log(\varepsilon)|^{-1}$ and recalling the definition of $U_{\varepsilon}$, we get
\begin{multline}\label{modest}
	\langle\xi\rangle^{2}|\widehat{u}_{\varepsilon}(t,\xi)|^{2}+|\partial_{t}\widehat{u}_{\varepsilon}(t,\xi)|^{2}\\\lesssim \varepsilon^{-2L_{1}-L_{2}-1}\left(\langle\xi\rangle^{2}|\widehat{u}_{0}(\xi)|^{2}+|\widehat{u}_{1}(\xi)|^{2}+\int_{0}^{T}|\widehat{f}_{\varepsilon}(\tau,\xi)|^{2}\mathrm{d}\tau\right).
\end{multline}
 Multiplying  by powers of $\langle\xi\rangle$ for any $s\in\mathbb{R}$ and using the $\mathcal{H}_{\hbar,V}$-Plancherel formula, we obtain
\begin{multline}\label{uuestt}
	\|u_{\varepsilon}(t,\cdot)\|^{2}_{\mathrm{H}_{\mathcal{H}_{\hbar,V}}^{1+s}}+\|\partial_{t}u_{\varepsilon}(t,\cdot)\|^{2}_{\mathrm{H}_{\mathcal{H}_{\hbar,V}}^{s}}\\\lesssim		\varepsilon^{-2L_{1}-L_{2}-1}\left(\|u_{0}\|^{2}_{\mathrm{H}_{\mathcal{H}_{\hbar,V}}^{1+s}}+\|u_{1}\|^{2}_{\mathrm{H}_{\mathcal{H}_{\hbar,V}}^{s}}+\|f_{\varepsilon}\|^{2}_{L^{2}([0,T];\mathrm{H}_{\mathcal{H}_{\hbar,V}}^{s})}\right),
\end{multline}
with the constant independent of $\hbar$ and $t\in[0,T].$ Since $(f_{\varepsilon})_{\varepsilon}$ is $L^{2}([0, T] ; \mathrm{H}_{\mathcal{H}_{\hbar,V}}^{s})$-moderate regularisation  of $f$, therefore there exist positive constants $L_{3}$ and $c>0$ such that
\begin{equation}\label{fl2}
	\|f_{\varepsilon}\|_{L^{2}([0,T];\mathrm{H}_{\mathcal{H}_{\hbar,V}}^{s})}\leq c\varepsilon^{-L_{3}}.
\end{equation}
On integrating the estimate \eqref{uuestt} with respect to the variable $t\in[0,T]$ and then 
 combining  together with \eqref{fl2}, we obtain
 \begin{equation}
 	\left\|u_{\varepsilon}\right\|_{L^{2}([0,T];\mathrm{H}_{\mathcal{H}_{\hbar,V}}^{1+s})} \lesssim \varepsilon^{-L_1-L_2-L_{3}}\text{ and } \left\|\partial_t u_{\varepsilon}\right\|_{L^{2}([0,T];\mathrm{H}_{\mathcal{H}_{\hbar,V}}^{s})} \lesssim \varepsilon^{-L_1-L_2-L_{3}-1}.
 \end{equation}
 Therefore, we deduce that 	$u_{\varepsilon} \text { is } L^{2}([0, T] ; \mathrm{H}_{\mathcal{H}_{\hbar,V}}^{1+s})$-moderate. This concludes the proof.
\end{proof}
Thus, we have proved the existence of very weak solution for the Cauchy problem \eqref{dispde1}. We will now prove the uniqueness of very weak solution in the sense of Definition \ref{uniquedef}.
\begin{proof}[Proof of Theorem \ref{uniq}]
Let $(u_{\varepsilon})_{\varepsilon}$	 and $(\tilde{u}_{\varepsilon})_{\varepsilon}$	 be the families of solutions corresponding to the Cauchy problems \eqref{reg1} and \eqref{reg2}, respectively. 
Denoting $w_{\varepsilon}(t,k):=u_{\varepsilon}(t,k)-\tilde{u}_{\varepsilon}(t,k)$, we get
\begin{equation}\label{weqn}
	\left\{\begin{array}{l}
		\partial_{t}^{2} w_{\varepsilon}(t, k)+a_{\varepsilon}(t)\mathcal{H}_{\hbar,V}w_{\varepsilon}(t, k)+q_{\varepsilon}(t)w_{\varepsilon}(t, k)=g_{\varepsilon}(t, k),\quad(t, k) \in(0, T] \times \hbar\mathbb{Z}^{n}, \\
		w_{\varepsilon}(0, k)=0,\quad k \in \hbar\mathbb{Z}^{n}, \\
		\partial_{t} w_{\varepsilon}(0, k)=0,\quad k \in \hbar\mathbb{Z}^{n},
	\end{array}\right.
\end{equation}
where 
\begin{equation}
	g_{\varepsilon}(t,k):=\left(\tilde{a}_{\varepsilon}-a_{\varepsilon}\right)(t) \mathcal{H}_{\hbar,V} \tilde{u}_{\varepsilon}(t, k)+\left(\tilde{q}_{\varepsilon}- q_{\varepsilon}\right)(t) \tilde{u}_{\varepsilon}(t, k)+(f_{\varepsilon}-\tilde{f}_{\varepsilon})(t,k).
\end{equation}
Since  $(a_{\varepsilon}-\tilde{a}_{\varepsilon})_{\varepsilon}$ is $L_{1}^{\infty}$-negligible,  $(q_{\varepsilon}-\tilde{q}_{\varepsilon})_{\varepsilon}$ is $L^{\infty}$-negligible, and $(f_{\varepsilon}-\tilde{f}_{\varepsilon})_{\varepsilon}$ is $L^{2}([0,T];\mathrm{H}^{s}_{\mathcal{H}_{\hbar,V}})$-negligible, it follows that $g_{\varepsilon}$ is $L^{2}([0,T];\mathrm{H}^{s}_{\mathcal{H}_{\hbar,V}})$-negligible.
 Taking the Fourier transform  with respect to $k\in\hbar\mathbb{Z}^{n}$ and then using the  transformation similar to Theorem \ref{ext}, the Cauchy problem \eqref{weqn} reduces to
\begin{equation}
	\left\{\begin{array}{l}
		W_{\varepsilon}^{\prime}(t,\xi)=i \langle\xi\rangle A_{\varepsilon}(t) W_{\varepsilon}(t,\xi)+i\langle\xi\rangle^{-1}Q_{\varepsilon}(t) W_{\varepsilon}(t,\xi)+G_{\varepsilon}(t,\xi),\quad\xi\in \mathcal{I}_{\hbar}, \\
		W_{\varepsilon}(0,\xi)=0,\quad \xi\in\mathcal{I}_{\hbar},
	\end{array}\right.
\end{equation}
where $A_{\varepsilon}(t),Q_{\varepsilon}(t)$ are given in \eqref{aepqep} and $G_{\varepsilon}(t,\xi)=[0,\widehat{g_{\varepsilon}}(t,\xi)]^{\mathrm{T}}$.
If we now define the energy
\begin{equation}
	E_{\varepsilon}(t,\xi):=(S_{\varepsilon}(t) W_{\varepsilon}(t,\xi), W_{\varepsilon}(t,\xi)),
\end{equation}
where $S_{\varepsilon}(t)$ is given by \eqref{sepl}, then similar to estimate \eqref{vvEest}, we have
\begin{equation}\label{wwEest}
	c_{0}\left|W_{\varepsilon}(t,\xi)\right|^{2} \leq E_{\varepsilon}(t,\xi) \leq \left(1+c_{1}\omega(\varepsilon)^{-L_{1}}\right)\left|W_{\varepsilon}(t,\xi)\right|^{2},
\end{equation} 
where $c_{0}$ and $c_{1}$ are positive constants.
Then we can calculate  
\begin{eqnarray}\label{etest}
	\partial_{t} E_{\varepsilon}(t, \xi)&\leq&\left(\left\|\partial_{t}S_{\varepsilon}(t)\right\|+\|\left(S_{\varepsilon} Q_{\varepsilon}-Q^{*}_{\varepsilon}S_{\varepsilon}\right)(t)\|\right)|W_{\xi}(t,\xi)|^{2}+\nonumber\\&&2\|S_{\varepsilon}(t)\||G_{\varepsilon}(t,\xi)||W_{\varepsilon}(t,\xi)|\nonumber\\
	 & \leq&\left(\left\|\partial_{t} S_{\varepsilon}(t)\right\|+\|\left(S_{\varepsilon} Q_{\varepsilon}-Q^{*}_{\varepsilon}S_{\varepsilon}\right)(t)\|+\left\|S_{\varepsilon}(t)\right\|\right) \left| W_{\varepsilon}(t, \xi)\right|^{2}+ \nonumber\\
	&&\left\|S_{\varepsilon}(t)\right\|\left|G_{\varepsilon}(t, \xi)\right|^{2}\nonumber\\
	&\leq&\left(1+|\partial_{t}a_{\varepsilon}(t)|+|q_{\varepsilon}(t)|+2|a_{\varepsilon}(t)|\right)\left| W_{\varepsilon}(t, \xi)\right|^{2}+(1+|a_{\varepsilon}(t)|)\left|G_{\varepsilon}(t, \xi)\right|^{2}.\nonumber\\
\end{eqnarray}
 Combining \eqref{std} and \eqref{wwEest} together with \eqref{etest}, and then using the Gronwall's lemma, we obtain the following estimate
\begin{multline}\label{pgron}
	|W_{\varepsilon}(t,\xi)|^{2}\leq c_{0}^{-1}C_{T}e^{\kappa_{T}\left(\omega(\varepsilon)^{-L_{1}-1}+\omega(\varepsilon)^{-L_{2}}+\omega(\varepsilon)^{-L_{1}}\right)}\left(\left|W_{\varepsilon}(0,\xi)\right|^{2}+\int_{0}^{T}|G_{\varepsilon}(\tau,\xi)|^{2}\mathrm{d}\tau\right),
\end{multline}
where constants are similar to estimate \eqref{vineq}.
Putting $\omega(\varepsilon)\sim|\log(\varepsilon)|^{-1}$ and using the fact that  $W_{\varepsilon}(0,\xi)\equiv 0$ for all $\varepsilon\in(0,1]$, we get
\begin{equation}
	|W_{\varepsilon}(t,\xi)|^{2}\lesssim 		\varepsilon^{-2L_{1}-L_{2}-1}\int_{0}^{T}|G_{\varepsilon}(\tau,\xi)|^{2}\mathrm{d}\tau,\quad (t,\xi)\in[0,T]\times\mathcal{I}_{\hbar},
\end{equation}
with the constant independent of $t\in[0,T]$ and $\xi\in\mathcal{I}_{\hbar}$. Recalling the definition of $W_{\varepsilon}$, we get
\begin{equation}\label{modest}
	\langle\xi\rangle^{2}|\widehat{w}_{\varepsilon}(t,\xi)|^{2}+|\partial_{t}\widehat{w}_{\varepsilon}(t,\xi)|^{2}\lesssim 		\varepsilon^{-2L_{1}-L_{2}-1}\int_{0}^{T}|\widehat{g}_{\varepsilon}(\tau,\xi)|^{2}\mathrm{d}\tau.
\end{equation}
 Multiplying  by powers of $\langle\xi\rangle$ for any $s\in\mathbb{R}$ and using the $\mathcal{H}_{\hbar,V}$-Plancherel formula, we obtain
 \begin{equation}
 	\|w_{\varepsilon}(t,\cdot)\|^{2}_{\mathrm{H}_{\mathcal{H}_{\hbar,V}}^{1+s}}+\|\partial_{t}w_{\varepsilon}(t,\cdot)\|^{2}_{\mathrm{H}_{\mathcal{H}_{\hbar,V}}^{s}}\lesssim 		\varepsilon^{-2L_{1}-L_{2}-1}\|g_{\varepsilon}\|^{2}_{L^{2}([0,T];\mathrm{H}_{\mathcal{H}_{\hbar,V}}^{s})},
 \end{equation}
for all $t\in[0,T]$. Since
   $g_{\varepsilon}$ is $L^{2}([0,T];\mathrm{H}^{s}_{\mathcal{H}_{\hbar,V}})$-negligible, therefore we obtain 
\begin{equation}
	\|w_{\varepsilon}(t,\cdot)\|^{2}_{\mathrm{H}_{\mathcal{H}_{\hbar,V}}^{1+s}}+\|\partial_{t}w_{\varepsilon}(t,\cdot)\|^{2}_{\mathrm{H}_{\mathcal{H}_{\hbar,V}}^{s}}\lesssim 		\varepsilon^{-2L_{1}-L_{2}-1}\varepsilon^{2L_{1}+L_{2}+1+q}=\varepsilon^{q},\quad \text{ for all }q\in\mathbb{N}_{0},
\end{equation}
for all $t\in [0,T]$.
On integrating the above estimate with respect to the variable $t\in[0,T]$, we obtain
\begin{equation}
	\|w_{\varepsilon}\|^{2}_{L^{2}([0,T];\mathrm{H}_{\mathcal{H}_{\hbar,V}}^{1+s})}+\|\partial_{t}w_{\varepsilon}\|^{2}_{L^{2}([0,T];\mathrm{H}_{\mathcal{H}_{\hbar,V}}^{s})}\lesssim 	\varepsilon^{q},\quad \text{ for all } q\in\mathbb{N}_{0},
\end{equation}
with the constant independent of $\hbar$ and $t\in[0,T]$.  Thus $(u_{\varepsilon}-\tilde{u}_{\varepsilon})_{\varepsilon}$ is  $L^{2}([0,T];\mathrm{H}^{1+s}_{\mathcal{H}_{\hbar,V}})$-negligible. This completes the proof.
\end{proof}
Thus, we have proved that the very weak solution is unique in the sense of Definition \ref{uniquedef}. We will now prove the consistency of very weak solutions.
\begin{proof}[Proof of Theorem \ref{cnst}]
	Let $\tilde{u}$ be the classical solution given by Theorem \ref{class}. By definition, we know that
\begin{equation}\label{cnst1}
	\left\{\begin{array}{l}
		\partial_{t}^{2} \tilde{u}(t, k)+a(t) \mathcal{H}_{\hbar,V}\tilde{u}(t, k)+q(t) \tilde{u}(t, k)=f(t,k),\quad(t, k) \in(0, T] \times \hbar\mathbb{Z}^{n}, \\
		\tilde{u}(0, k)=u_{0}(k),\quad k \in \hbar\mathbb{Z}^{n}, \\
		\partial_{t} \tilde{u}(0, k)=u_{1}(k),\quad k \in \hbar\mathbb{Z}^{n},
	\end{array}\right.
\end{equation} 	
and there exists a net $(u_{\varepsilon})_{\varepsilon}$   such that
\begin{equation}\label{cnst2}
	\left\{\begin{array}{l}
		\partial_{t}^{2} u_{\varepsilon}(t, k)+a_{\varepsilon}(t) \mathcal{H}_{\hbar,V}u_{\varepsilon}(t, k)+q_{\varepsilon}(t) u_{\varepsilon}(t, k)=f_{\varepsilon}(t,k),\quad(t, k) \in(0, T] \times \hbar\mathbb{Z}^{n}, \\
		u_{\varepsilon}(0, k)=u_{0}(k),\quad k \in \hbar\mathbb{Z}^{n}, \\
		\partial_{t} u_{\varepsilon}(0, k)=u_{1}(k),\quad k \in \hbar\mathbb{Z}^{n}.
	\end{array}\right.
\end{equation} 
	Observing that the nets $\left(a_{\varepsilon}-a\right)_{\varepsilon},\left(q_{\varepsilon}-q\right)_{\varepsilon}$ and $\left(f_{\varepsilon}-f\right)_{\varepsilon}$ are converging to  $0$ for $a\in L^{\infty}_{1}([0,T])$, $b\in L^{\infty}([0,T])$ and $f\in L^{2}([0,T];\mathrm{H}^{s}_{\mathcal{H}_{\hbar,V}})$, we can rewrite \eqref{cnst1} as
\begin{equation}\label{cnst3}
	\left\{\begin{array}{l}
		\partial_{t}^{2} \tilde{u}(t, k)+a_{\varepsilon}(t) \mathcal{H}_{\hbar,V}\tilde{u}(t, k)+q_{\varepsilon}(t) \tilde{u}(t, k)=f_{\varepsilon}(t,k)+g_{\varepsilon}(t,k),~~ (t, k) \in(0, T] \times \hbar\mathbb{Z}^{n}, \\
		\tilde{u}(0, k)=u_{0}(k),\quad k \in \hbar\mathbb{Z}^{n}, \\
		\partial_{t} \tilde{u}(0, k)=u_{1}(k),\quad k \in \hbar\mathbb{Z}^{n},
	\end{array}\right.
\end{equation} 	
where $$g_{\varepsilon}(t,k):=\left(a_{\varepsilon}-a\right)(t) \mathcal{H}_{\hbar,V} \tilde{u}(t, k)+\left(q_{\varepsilon}-q\right)(t) \tilde{u}(t, k)+\left(f-f_{\varepsilon}\right)(t,k),$$
  $g_{\varepsilon} \in L^{2}([0, T] ; \mathrm{H}_{\mathcal{H}_{\hbar,V}}^s)$ and  $g_{\varepsilon}\to0$ in $L^{2}([0, T] ; \mathrm{H}_{\mathcal{H}_{\hbar,V}}^s)$ as $\varepsilon \rightarrow 0$. From \eqref{cnst2} and \eqref{cnst3}, we get that  $w_{\varepsilon}(t,k):=\left(\tilde{u}-u_{\varepsilon}\right)(t,k)$  solves the Cauchy problem
\begin{equation}\label{cnst4}
	\left\{\begin{array}{l}
		\partial_{t}^{2} w_{\varepsilon}(t, k)+a_{\varepsilon}(t) \mathcal{H}_{\hbar,V}w_{\varepsilon}(t, k)+q_{\varepsilon}(t) w_{\varepsilon}(t, k)=g_{\varepsilon}(t,k),\quad(t, k) \in(0, T] \times \hbar\mathbb{Z}^{n}, \\
		w_{\varepsilon}(0, k)=0,\quad k \in \hbar\mathbb{Z}^{n}, \\
		\partial_{t} w_{\varepsilon}(0, k)=0,\quad k \in \hbar\mathbb{Z}^{n}.
	\end{array}\right.
\end{equation}
Similar to the proof of Theorem \ref{uniq}, the following energy estimate can be easily obtained
\begin{eqnarray}
	\partial_t E_{\varepsilon}(t, \xi) 
	&\leq&\left(\left\|\partial_{t}S_{\varepsilon}(t)\right\|+\|\left(S_{\varepsilon} Q_{\varepsilon}-Q^{*}_{\varepsilon}S_{\varepsilon}\right)(t)\|+\|S_{\varepsilon}(t)\|\right)|W_{\varepsilon}(t,\xi)|^{2}+\nonumber\\
	&&\|S_{\varepsilon}(t)\||G_{\varepsilon}(t,\xi)|^{2}\nonumber\\
	&\leq &\left(1+|a_{\varepsilon}^{\prime}(t)|+|q_{\varepsilon}(t)|+2|a_{\varepsilon}(t)|\right)|W_{\varepsilon}(t,\xi)|^{2}+\nonumber\\
	&&(1+|a_{\varepsilon}(t)|)|G_{\varepsilon}(t,\xi)|^{2}.
\end{eqnarray}
The coefficients are sufficiently regular, so we simply obtain
\begin{equation}
\partial_t E_{\varepsilon}(t, \xi) \leq c_1E_{\varepsilon}(t, \xi)+c_2\left|G_{\varepsilon}(t, \xi)\right|^{2},
\end{equation}
for some  positive constants $c_{1}$ and $c_{2}$. Then using the Gronwall's lemma and the energy bounds similar to estimate \eqref{vEest}, we obtain
\begin{equation}
	|W_{\varepsilon}(t,\xi)|^{2}\lesssim |W_{\varepsilon}(0,\xi)|^{2}+\int_{0}^{T}|G_{\varepsilon}(\tau,\xi)|^{2}\mathrm{d}\tau,
\end{equation}
with the constant independent of  $t\in[0,T]$ and $\xi\in\mathcal{I}_{\hbar}$. Using the Plancherel formula and the fact that $W_{\varepsilon}(0,\xi)\equiv 0$ for all $\varepsilon\in(0,1]$, we obtain 
\begin{equation}
	\|w_{\varepsilon}(t,\cdot)\|^{2}_{\mathrm{H}_{\mathcal{H}_{\hbar,V}}^{1+s}}+\|\partial_{t}w_{\varepsilon}(t,\cdot)\|^{2}_{\mathrm{H}_{\mathcal{H}_{\hbar,V}}^{s}}\lesssim 	\|g_{\varepsilon}\|^{2}_{L^{2}([0,T];\mathrm{H}_{\mathcal{H}_{\hbar,V}}^{s})},
\end{equation}
for all $t\in [0,T]$.
On integrating the above estimate with respect to the variable $t\in[0,T]$, we obtain
\begin{equation}
	\|w_{\varepsilon}\|^{2}_{L^{2}([0,T];\mathrm{H}_{\mathcal{H}_{\hbar,V}}^{1+s})}+\|\partial_{t}w_{\varepsilon}\|^{2}_{L^{2}([0,T];\mathrm{H}_{\mathcal{H}_{\hbar,V}}^{s})}\lesssim 	\|g_{\varepsilon}\|^{2}_{L^{2}([0,T];\mathrm{H}_{\mathcal{H}_{\hbar,V}}^{s})},
\end{equation}
with the constant independent of $\hbar$ and $t\in[0,T]$.
Since
$g_{\varepsilon} \rightarrow 0$ in $L^{2}([0, T] ; \mathrm{H}_{\mathcal{H}_{\hbar,V}}^s)$, therefore we have
\begin{equation}
w_{\varepsilon} \to 0 \text{ in }L^{2}([0, T]; \mathrm{H}_{\mathcal{H}_{\hbar,V}}^{1+s}),\quad \varepsilon\to 0,
\end{equation}
 i.e.,
\begin{equation}
	u_{\varepsilon} \to \tilde{u} \text{ in }L^{2}([0, T]; \mathrm{H}_{\mathcal{H}_{\hbar,V}}^{1+s}) ,\quad \varepsilon\to 0.
\end{equation} 
Furthermore, the limit is the same for every representation of $u$, since they will differ from $\left(u_{\varepsilon}\right)_{\varepsilon}$ by a $L^{2}([0, T]; \mathrm{H}_{\mathcal{H}_{\hbar,V}}^{1+s})$-negligible net. This concludes the proof.
\end{proof}
\section{Semiclassical limit as $\hbar \to 0$}\label{sclass}
In this section we will consider the semiclassical limit of solutions as $\hbar\to 0$. 
\begin{proof}[Proof of Theorem \ref{semlimit}]
	Consider two Cauchy problems:
	\begin{equation}\label{CP1}
		\left\{\begin{array}{l}
			\partial_{t}^{2} u(t, k)+a(t)\mathcal{H}_{\hbar,V}u(t, k)+q(t) u(t, k)=f(t,k), \quad  (t,k) \in(0,T]\times\hbar\mathbb{Z}^{n}, \\
			u(0, k)=u_{0}(k), \quad k\in\hbar\mathbb{Z}^{n},\\
			\partial_{t} u(0, k)=u_{1}(k),\quad k\in\hbar\mathbb{Z}^{n},
		\end{array}\right.
	\end{equation}
	and
	\begin{equation}\label{CP2}
		\left\{
		\begin{array}{ll}
			\partial^{2}_{t}v(t,x)+a(t)\mathcal{H}_{V}v(t,x)+q(t)v(t,x)=f(t,x), \quad (t,x) \in(0,T]\times\mathbb{R}^{n},\\
			v(0,x)=u_{0}(x),\quad x\in\mathbb{R}^{n},\\
			\partial_{t}v(0,x)=u_1(x),\quad x\in\mathbb{R}^{n},
		\end{array}
		\right.
	\end{equation}
	where $\mathcal{H}_{V}$ is the usual Schr\"{o}dinger operator on $\mathbb{R}^{n}$. The potential in the discrete Schr\"{o}dinger operator is the restriction of the potential in the usual Schr\"{o}dinger operator to $\hbar\mathbb{Z}^{n}$. Here the initial data and the source term of the Cauchy problem \eqref{CP1} is the evaluation of the initial data and the source term from \eqref{CP2} on the lattice $\hbar\mathbb{Z}^{n}$. From the equations \eqref{CP1} and \eqref{CP2}, 
	denoting $w:=u-v$,  we get
	\begin{equation}\label{CPF}
		\left\{
		\begin{array}{ll}
			\partial^{2}_{t}w(t,k)+a(t)\mathcal{H}_{\hbar,V}w(t,k)+q(t) w(t,k)=a(t)\left(\mathcal{H}_{V}-\mathcal{H}_{\hbar,V}\right)v(t,k),~~  k\in\hbar\mathbb{Z}^{n},\\
			w(0,k)=0,\quad k\in\hbar\mathbb{Z}^{n},\\
			\partial_{t}w(0,k)=0,\quad k\in\hbar\mathbb{Z}^{n}.               
		\end{array}		\right.
	\end{equation}
	Since $w_{0}=w_{1}=0$, applying Theorem \ref{class} for the above Cauchy problem and using estimate \eqref{uestt}, we get   
	\begin{eqnarray}\label{cgt1}
		\|w(t,\cdot)\|^{2}_{\mathrm{H}_{\mathcal{H}_{\hbar,V}}^{1+s}}+\|w_{t}(t,\cdot)\|^{2}_{\mathrm{H}_{\mathcal{H}_{\hbar,V}}^{s}}
		&\leq& C_{T}\|a\left(\mathcal{H}_{V}-\mathcal{H}_{\hbar,V}\right)v\|^{2}_{L^{2}([0,T];\mathrm{H}_{\mathcal{H}_{\hbar,V}}^{s})}\nonumber\\
		&\leq& C_{T}\|a\|^{2}_{L^{\infty}([0,T])}\|\left(\mathcal{H}_{V}-\mathcal{H}_{\hbar,V}\right)v\|^{2}_{L^{2}([0,T];\mathrm{H}_{\mathcal{H}_{\hbar,V}}^{s})},\nonumber\\
	\end{eqnarray}
where the constant $C_{T}$ is given by
\begin{equation}
	C_{T}=c_{0}^{-1}(1+\left\|a\right\|_{L^{\infty}})e^{c_{0}^{-1}\left(1+\|a^{\prime}\|_{L^{\infty}}+\|q\|_{L^{\infty}}+2\left\|a\right\|_{L^{\infty}}\right)T},
\end{equation}
with $c_{0}=\min\{a_{0},1\}$.
	
	Now we will estimate the term $\left\|\left(\mathcal{H}_{V}-\mathcal{H}_{\hbar,V}\right)v\right\|^{2}_{L^{2}([0,T];\mathrm{H}_{\mathcal{H}_{\hbar,V}}^{s})}$. 
	Let $\phi\in C^{4}(\mathbb{R}^{n})$, then by Taylor's theorem with the Lagrange's form of the remainder, we have
	\begin{equation}\label{tylr}
		\phi(\xi+\mathbf{h})=\sum_{|\alpha|\leq 3} \frac{\partial^{\alpha} \phi(\xi)}{\alpha !} \mathbf{h}^{\alpha}+\sum_{|\alpha|=4} \frac{\partial^{\alpha} \phi(\xi+\theta_{\xi}\mathbf{h})}{\alpha !} \mathbf{h}^{\alpha},
	\end{equation} 
	for some $\theta_{\xi}\in (0,1)$ depending on $\xi$.  Let $v_j$ be the $j^{th}$ basis vector in $\Zn$, having all zeros except for $1$ as the $j^{th}$ component and then by taking  $\mathbf{h}=v_{j}$ and $-v_{j}$ in \eqref{tylr}, we have
	\begin{eqnarray}\label{d1}
		\phi(\xi+v_{j})=\phi(\xi)+\phi^{(v_{j})}(\xi)+\frac{1}{2!}\phi^{(2v_{j})}(\xi)+\frac{1}{3!}\phi^{(3v_{j})}(\xi)+\frac{1}{4!}\phi^{(4v_{j})}(\xi+\theta_{j,\xi} v_{j}),
	\end{eqnarray} and
	\begin{eqnarray}\label{d2}
		\phi(\xi-v_{j})=\phi(\xi)-\phi^{(v_{j})}(\xi)+\frac{1}{2!}\phi^{(2v_{j})}(\xi)-\frac{1}{3!}\phi^{(3v_{j})}(\xi)+\frac{1}{4!}\phi^{(4v_{j})}(\xi-\tilde{\theta}_{j,\xi} v_{j}),\end{eqnarray}
	for some $\theta_{j,\xi},\tilde{\theta}_{j,\xi}\in(0,1)$. Using \eqref{d1} and \eqref{d2}, we have
	\begin{equation*}
		\phi(\xi+v_{j})+\phi(\xi-v_{j})-2\phi(\xi)=\phi^{(2v_{j})}(\xi)+\frac{1}{4!}\left(\phi^{(4v_{j})}(\xi+\theta_{j,\xi} v_{j})+\phi^{(4v_{j})}(\xi-\tilde{\theta}_{j,\xi} v_{j})\right).
	\end{equation*}
	Since $\delta_{\xi_{j}}^{2}\phi(\xi)=\phi(\xi+v_{j})+\phi(\xi-v_{j})-2\phi(\xi)$, where $\delta_{\xi_{j}}\phi(\xi):=\phi(\xi+\frac{1}{2}v_{j})-\phi(\xi-\frac{1}{2}v_{j}),$ is the usual central difference operator, it follows that
	\begin{eqnarray}\label{d3}
		\delta_{\xi_{j}}^{2}\phi(\xi)&=&\phi^{(2v_{j})}(\xi)+\frac{1}{4!}\left(\phi^{(4v_{j})}(\xi+\theta_{j,\xi} v_{j})+\phi^{(4v_{j})}(\xi-\tilde{\theta}_{j,\xi}v_{j})\right).
	\end{eqnarray}
	Now by adding all the above $n$-equations for $j=1,\dots,n$, we get
	\begin{eqnarray}\label{d3}
		\sum\limits_{j=1}^{n}\delta_{\xi_{j}}^{2}\phi(\xi)=\sum\limits_{j=1}^{n}\phi^{(2v_{j})}(\xi)+\frac{1}{4!}\sum\limits_{j=1}^{n}\left(\phi^{(4v_{j})}(\xi+\theta_{j,\xi} v_{j})+\phi^{(4v_{j})}(\xi-\tilde{\theta}_{j,\xi} v_{j})\right).
	\end{eqnarray}
	Let us define a translation operator $E_{\theta_{j}v_{j}}\phi:\mathbb{R}^{n}\to \mathbb{R}$  by $E_{\theta_{j}v_{j}}\phi(\xi):=\phi(\xi-\theta_{j,\xi}v_{j}),$ then we get
	\begin{eqnarray}
		\sum\limits_{j=1}^{n}\delta_{\xi_{j}}^{2}\phi(\xi)-\sum\limits_{j=1}^{n}\frac{\partial^{2}}{\partial\xi_{j}^{2}}\phi(\xi)=\frac{1}{4!}\sum\limits_{j=1}^{n}\left(E_{-\theta_{j}v_{j}}\phi^{(4v_{j})}(\xi)+E_{\tilde{\theta}_{j}v_{j}}\phi^{(4v_{j})}(\xi)\right).
	\end{eqnarray}
	Now we extend this to $\hbar\Zn$. Consider a function  $\phi_{\hbar}:\mathbb{R}^{n}\to \mathbb{R}$  defined by
	$\phi_{\hbar}(\xi):=\phi(\hbar\xi)$. Clearly  $\phi_{\hbar}\in C^{4}(\mathbb{R}^{n})$, if we take $\phi\in C^{4}(\mathbb{R}^{n})$. Now we have
	\begin{eqnarray}\label{eqhzn}
		\mathcal{L}_{1}\phi_{\hbar}(\xi)-\mathcal{L}\phi_{\hbar}(\xi)=\frac{1}{4!}\sum\limits_{j=1}^{n}\left(E_{-\theta_{j}v_{j}}\phi_{\hbar}^{(4v_{j})}(\xi)+E_{\tilde{\theta}_{j}v_{j}}\phi_{\hbar}^{(4v_{j})}(\xi)\right),
	\end{eqnarray}
	where $\mathcal{L}$ is the Laplacian on $\Rn$ and $\mathcal{L}_{1}$ is the discrete difference Laplacian on $\Zn$.  One can quickly notice that 
	\begin{equation}
		E_{-\theta_{j}v_{j}}\phi_{\hbar}^{(4v_{j})}(\xi)=\phi_{\hbar}^{(4v_{j})}(\xi+\theta_{j,\xi} v_{j})=\hbar^{4}\phi^{(4v_{j})}(\hbar\xi+\hbar\theta_{j,\xi} v_{j})=\hbar^{4}E_{-\hbar\theta_{j}v_{j}}\phi^{(4 v_{j})}(\hbar\xi).
	\end{equation}
	Therefore, the equality \eqref{eqhzn} becomes
	\begin{eqnarray}\label{lapdif}
		\left(\mathcal{L}_{\hbar}-\hbar^{2}\mathcal{L}\right)\phi(\hbar\xi)=\frac{\hbar^{4}}{4!}\sum\limits_{j=1}^{n}\left(E_{-\hbar\theta_{j}v_{j}}\phi^{(4v_{j})}(\hbar\xi)+E_{\hbar\tilde{\theta}_{j}v_{j}}\phi^{(4v_{j})}(\hbar\xi)\right).
	\end{eqnarray}
	Combining \eqref{dhamil}, \eqref{eucd} and \eqref{lapdif}, we get 
	\begin{multline}
		\left(\mathcal{H}_{V}-\mathcal{H}_{\hbar,V}\right)\phi(\hbar\xi)=  	\left(\hbar^{-2}\mathcal{L}_{\hbar}-\mathcal{L}\right)\phi(\hbar\xi)=\frac{\hbar^{2}}{4!}\sum\limits_{j=1}^{n}\left(E_{-\hbar\theta_{j}v_{j}}\phi^{(4v_{j})}(\hbar\xi)\right.+\\
		\left.E_{\hbar\tilde{\theta}_{j}v_{j}}\phi^{(4v_{j})}(\hbar\xi)\right).
	\end{multline}
	Hence, it follows that
	\begin{equation}\label{EQ:convh}
		\left\|\left(\mathcal{H}_{V}-\mathcal{H}_{\hbar,V}\right)\phi\right\|^{2}_{\mathrm{H}_{\mathcal{H}_{\hbar,V}}^{s}}\lesssim\hbar^{4}\max_{1\leq j\leq n}\left(\left\|E_{-\hbar\theta_{j}v_{j}}\phi^{(4v_{j})}\right\|^{2}_{\mathrm{H}_{\mathcal{H}_{\hbar,V}}^{s}}
		+\left\|E_{\hbar\tilde{\theta}_{j}v_{j}}\phi^{(4v_{j})}\right\|^{2}_{\mathrm{H}_{\mathcal{H}_{\hbar,V}}^{s}}\right).
	\end{equation}
Combining the relation \eqref{semb} with the Sobolev embedding theorem (see e.g.\cite[Excercise 2.6.17]{Ruzhansky-Turunen:BOOK}), we have
\begin{equation}\label{sobemb}
s>k+\frac{n}{2} \Longrightarrow 	\mathrm{H}_{\mathcal{H}_{V}}^{s}(\mathbb{R}^{n}) \subseteq C^k\left(\mathbb{R}^n\right).
\end{equation}
Since $(u_{0}, u_{1}) \in \mathrm{H}_{\mathcal{H}_{V}}^{1+s}\times \mathrm{H}_{\mathcal{H}_{V}}^{s}$ with $s>4+\frac{n}{2}$, therefore  using Theorem \ref{eucclass}, it follows that the classical solution satisfies $v\in C([0,T];\mathrm{H}_{\mathcal{H}_{V}}^{1+s})\cap C^{1}([0,T];\mathrm{H}_{\mathcal{H}_{V}}^{s})$ with $s>4+\frac{n}{2}$. Using the embedding \eqref{sobemb}, we get $v\in C^{4}(\mathbb{R}^{n})$ and using the hypothesis $(u_{0}^{(4v_{j})},u_{1}^{(4v_{j})})\in \mathrm{H}_{\mathcal{H}_{V}}^{1+s}\times \mathrm{H}_{\mathcal{H}_{V}}^{s}$ for all $j=1,\dots,n$, we deduce that
\begin{equation}\label{v4der}
	v^{(4v_{j})}(t,\cdot)\in \mathrm{H}_{\mathcal{H}_{V}}^{s}(\Rn), \quad \text{for all }t\in[0,T].
\end{equation} 
	Now from  \eqref{EQ:convh}, it follows that
	\begin{multline}\label{EQ:cht}
		\left\|\left(\mathcal{H}_{V}-\mathcal{H}_{\hbar,V}\right)v\right\|^{2}_{_{L^{2}([0,T];\mathrm{H}_{\mathcal{H}_{\hbar,V}}^{s})}}\lesssim\hbar^{4}\max_{1\leq j\leq n}\left(\left\|E_{-\hbar\theta_{j}v_{j}}v^{(4v_{j})}\right\|^{2}_{_{L^{2}([0,T];\mathrm{H}_{\mathcal{H}_{\hbar,V}}^{s})}}+\right.\\
		\left.\left\|E_{\hbar\tilde{\theta}_{j}v_{j}}v^{(4v_{j})}\right\|^{2}_{_{L^{2}([0,T];\mathrm{H}_{\mathcal{H}_{\hbar,V}}^{s})}}\right).
	\end{multline}
	Using \eqref{cgt1}, \eqref{v4der} and \eqref{EQ:cht}, we get $\|w(t,\cdot)\|^{2}_{\mathrm{H}_{\mathcal{H}_{\hbar,V}}^{1+s}}+\|w_{t}(t,\cdot)\|^{2}_{\mathrm{H}_{\mathcal{H}_{\hbar,V}}^{s}}\to 0$ as $\hbar\to 0.$ Hence $\|w(t,\cdot)\|_{\mathrm{H}_{\mathcal{H}_{\hbar,V}}^{1+s}} \to 0$ and $\|w_{t}(t,\cdot)\|_{\mathrm{H}_{\mathcal{H}_{\hbar,V}}^{s}} \to 0$ as $\hbar \to 0$. This concludes the proof of Theorem \ref{semlimit}.	
\end{proof}
We can prove Theorem \ref{vvyksemlimit} without making any substantial modifications to the proof of Theorem \ref{semlimit}.\\

\textbf{Acknowledgment}\\
\medskip
The authors would like to thanks Prof. M. Krishna for his  insightful comments.
\bibliographystyle{alphaabbr}
\bibliography{Discrete_Schrodinger}

\end{document}